\documentclass{article}
\usepackage{amsmath,amssymb,amsthm,color,graphicx,diagbox,tikz} 

\title{A variety of partially conservative sentences}
\author{Haruka Kogure\footnote{Email: kogure@stu.kobe-u.ac.jp}
\footnote{Graduate School of System Informatics, Kobe University, 1-1 Rokkodai, Nada, Kobe 657-8501, Japan.}
and Taishi Kurahashi\footnote{Email: kurahashi@people.kobe-u.ac.jp}
\footnote{Graduate School of System Informatics, Kobe University, 1-1 Rokkodai, Nada, Kobe 657-8501, Japan.}}
\date{}

\theoremstyle{plain}
\newtheorem{thm}{Theorem}[section]
\newtheorem*{thm*}{Theorem}
\newtheorem{lem}[thm]{Lemma}
\newtheorem{prop}[thm]{Proposition}
\newtheorem{cor}[thm]{Corollary}

\newtheorem*{fact*}{Fact}
\newtheorem{prob}[thm]{Problem}
\newtheorem*{prob*}{Problem}
\newtheorem*{cl*}{Claim}
\newtheorem*{scl*}{Subclaim}

\theoremstyle{definition}
\newtheorem{defn}[thm]{Definition}

\newtheorem{rem}[thm]{Remark}

\newcommand{\PA}{\mathsf{PA}}
\newcommand{\PR}{\mathrm{Pr}}
\newcommand{\Prf}{\mathrm{Prf}}

\newcommand{\Con}{\mathrm{Con}}

\newcommand{\gn}[1]{\ulcorner#1\urcorner}

\newcommand{\num}{\overline}
\newcommand{\True}{\mathrm{True}}

\newcommand{\SP}{\Sigma_n {\downarrow} \Pi_n}
\newcommand{\Cons}{\mathrm{Cons}}
\newcommand{\HCons}{\mathrm{HCons}}
\newcommand{\DCons}{\mathrm{DCons}}
\newcommand{\HDCons}{\mathrm{HDCons}}
\newcommand{\Th}{\mathrm{Th}}

\newcommand{\Bend}{\mathrm{C}^{-}}
\newcommand{\Ben}{\mathrm{B}}
\newcommand{\Benp}{\mathrm{C}}
\newcommand{\Bena}{\mathrm{C}^{\land}}
\newcommand{\Benb}{\mathrm{C}^{\mathcal{B}}}

\begin{document}

\maketitle

\begin{abstract}
We study the existence of a $\Theta$ sentence which is simultaneously $\Gamma$-conservative over consistent RE extensions $T$ and $U$ of Peano Arithmetic for various reasonable pairs $(\Gamma, \Theta)$. 
As a result of this study, we prove the existence of a sentence which is essentially $\Theta$ and exactly hereditarily $\Gamma$-conservative over any single theory $T$ for various reasonable pairs $(\Gamma, \Theta)$. 
This is an affirmative answer to Guaspari's question (\cite[Question (2) on p.~62]{Gua}). 
\end{abstract}

\section{Introduction}

After G\"odel's first incompleteness theorem, there have been a number of studies on the existence of independent sentences with various properties. 
For example, Mostowski \cite{Mos} proved that for any RE sequence $\langle T_i \mid i \in \omega \rangle$ of consistent RE theories extending Peano Arithmetic $\PA$, there exists a $\Pi_1$ sentence which is simultaneously independent over each $T_i$. 
Partially conservative sentences are also other examples of independent sentences, which have been naturally analyzed in the context of the incompleteness theorems.
Kreisel \cite{Kre} proved that the $\Sigma_1$ sentence $\neg \Con_T$, where $\Con_T$ is the standard consistency statement of an RE consistent extension $T$ of $\PA$, is $\Pi_1$-conservative over $T$. 
That is, for any $\Pi_1$ sentence $\pi$, if $T + \neg \Con_T \vdash \pi$, then $T \vdash \pi$. 
This is a strengthening of G\"odel's second incompleteness theorem. 
Of course, if $T$ is $\Sigma_1$-sound, then $\Con_T$ is independent over $T$. 
Furthermore, Smory\'nski \cite{Smo} proved that $T$ is $\Sigma_1$-sound if and only if $\Con_T$ is $\Sigma_1$-conservative over $T$. 

The systematic analysis of partially conservative sentences was founded by H\'ajek \cite{Haj79}, Guaspari \cite{Gua}, Lindstr\"om \cite{Lin84}, and others. 
Some of the results of these studies can be found in the textbooks by Lindstr\"om \cite{Lin} and H\'ajek and Pudl\'ak \cite{HP}.
Among other things, Guaspari \cite[Theorem 2.6]{Gua} proved that if $\Gamma$ and $\Theta$ range over $\{\Sigma_n, \Pi_n \mid n \geq 1\}$ and $\Theta \supseteq \Gamma^d$, then there exists a sentence which is essentially $\Theta$ and exactly hereditarily $\Gamma$-conservative over $T$ (See Theorem \ref{Thm_Gua}). 
Here, $\Sigma_n^d = \Pi_n$ and $\Pi_n^d = \Sigma_n$. 
Then, Guaspari \cite[p.~62]{Gua} asked the following two questions: 
\begin{enumerate}
    \item For any RE sequence $\langle T_i \mid i \in \omega \rangle$ of RE theories, is there a $\Gamma$ sentence which is independent and $\Gamma^d$-conservative over each $T_i$? 

    \item For any reasonable combination $(\Gamma, \Theta)$ of $\Pi$, $\Sigma$, and $\Delta$, does there exist a sentence which is essentially $\Theta$ and exactly hereditarily $\Gamma$-conservative over $T$?
\end{enumerate}

The first question has been studied by Misercque \cite{Mis,MisPhD}, Bennet \cite{Ben86,Ben}, and Kurahashi, Okawa, Shavrukov and Visser \cite{KOSV} (KOSV for short), and especially for finite number of theories, a solution has already been established.
For $\Gamma = \Sigma_n$, Bennet \cite{Ben86,Ben} found a necessary and sufficient condition of two theories $T$ and $U$ for the existence of a $\Pi_n$ sentence which is simultaneously independent and $\Sigma_n$-conservative over $T$ and $U$ (cf.~Theorem \ref{Thm_Ben}). 
For finite number of theories, a necessary and sufficient condition for the existence of a $\Pi_n$ sentence of this kind over those theories is given by KOSV \cite{KOSV}. 
On the other hand, for $\Gamma = \Pi_n$, KOSV also proved that for any finite family of theories, there always exists a $\Sigma_n$ sentence which is simultaneously independent and $\Pi_n$-conservative over those theories (cf.~Theorem \ref{Thm_KOSV}). 

Our main purpose of the present paper is to continue Bennet and KOSV's works and to study the second question of Guaspari.
We newly introduce the notion of $\SP$-conservativity, which is useful in the study of hereditary double $\Pi_n$-conservativity. 
We then examine the status of the following four notions for two theories $T$ and $U$ and each $\Gamma$ in the situation where $\Gamma$ ranges over $\Delta_n$, $\SP$, $\Sigma_n$, $\Pi_n$, $\Sigma_n \land \Pi_n$, and $\mathcal{B}(\Sigma_n)$ for $n \geq 1$: 
\begin{itemize}
    \item Simultaneous non-trivial hereditary $\Gamma$-conservativity over $T$ and $U$. 
    \item Simultaneous exact hereditary $\Gamma$-conservativity over $T$ and $U$. 
    \item Simultaneous non-trivial $\Gamma$-conservativity over $T$ and $U$. 
    \item Simultaneous exact $\Gamma$-conservativity over $T$ and $U$. 
\end{itemize}
Our results on this study are summarized in Section \ref{Sec:Summary}. 
As a consequence of this study, we prove that for any RE consistent extension $T$ of $\PA$ and any reasonable pair $(\Gamma, \Theta)$, there exists a sentence which is essentially $\Theta$ and exactly hereditarily $\Gamma$-conservative over $T$ (Theorem \ref{MT}). 
This is our solution to Guaspari's second problem in a general setting. 
In the context of the first incompleteness theorem, our results show that a wide variety of independent sentences indeed exist.

The organization of the present paper is as follows. 
In Section \ref{Sec:Pre}, we introduce some notions and facts used in our arguments. 
In Section \ref{Sec:Cons}, we review the known results on partially conservative sentences that are related to the present paper.  
Section \ref{Sec:Sigma_Pi} is devoted to studying $\Sigma_n$-conservativity and $\Pi_n$-conservativity. 
In Section \ref{Sec:SP}, we introduce the notion of $\SP$-conservativity and study this new notion. 
The notions of $\Delta_n$-conservativity, $\mathcal{B}(\Sigma_n)$-conservativity, and $\Sigma_n \land \Pi_n$-conservativity are analyzed in Sections \ref{Sec:Delta}, \ref{Sec:Boole}, and \ref{Sec:and}, respectively. 
Our results obtained in Sections from \ref{Sec:Sigma_Pi} to \ref{Sec:and} are summarized in Section \ref{Sec:Summary}. 
Finally, we prove our solution to Guaspari's question in Section \ref{Sec:answer}.

\section{Preliminaries}\label{Sec:Pre}

Throughout the present paper, we assume that $T$ and $U$ always denote any computable consistent extensions of Peano Arithmetic $\PA$ in the language $\mathcal{L}_A$ of first-order arithmetic. 
Let $\omega$ be the set of all natural numbers.
For each $k \in \omega$, let $\num{k}$ denote the numeral of $n$. 
For each formula $\varphi$, let $\gn{\varphi}$ be the numeral of the fixed G\"{o}del number of $\varphi$.

For each $n \in \omega$, we define the formula classes $\Sigma_n$ and $\Pi_n$ as usual.
A formula is called $\Sigma_n \wedge \Pi_n$ (resp.~$\Sigma_n \lor \Pi_n$) iff it is of the form $\sigma \wedge \pi$ (resp.~$\sigma \lor \pi$) for some $\Sigma_n$ formula $\sigma$ and $\Pi_n$ formula $\pi$. 
A formula is called $\Delta_n(T)$ iff it is $T$-provably equivalent to both some $\Sigma_n$ formula and some $\Pi_n$ formula.
A formula is called $\Sigma_n(T)$ (resp.~$\Pi_n(T)$) iff it is $T$-provably equivalent to some $\Sigma_n$ (resp.~$\Pi_n$) formula. 

For each $\Gamma \in \{\Sigma_n, \Pi_n \mid n \geq 1\}$, 
let $\Gamma(x)$ be a $\Delta_1(\PA)$ formula naturally saying that ``$x$ is a $\Gamma$ formula''  and $\True_{\Gamma}(x)$ be a $\Gamma$ formula expressing that ``$x$ is a true $\Gamma$ sentence''. 
The formula $\True_{\Gamma}(x)$ satisfies the condition that for any $\Gamma$ sentence $\varphi$,
we have $\PA \vdash \varphi \leftrightarrow \True_{\Gamma}(\gn{\varphi})$
(See H\'{a}jek and Pudl\'ak \cite{HP}).

Let $\Prf_T(x,y)$ be a $\Delta_1(\PA)$ formula naturally expressing that ``$y$ is a $T$-proof of $x$''. 
Let $\PR_T(x)$ be the $\Sigma_1$ formula $\exists y \Prf_T(x,y)$.
The formulas $\Prf_T(x,y)$ and $\PR_T(x)$ are called a \textit{proof predicate for} $T$ and a \textit{provability predicate for} $T$ respectively.

We introduce the witness comparison notation (cf.~Guaspari and Solovay \cite{GS}). 
For formulas $\varphi \equiv \exists x \alpha (x)$ and $\psi \equiv \exists x \beta(x)$,
\begin{itemize}
\item 
$\varphi \prec \psi : \equiv \exists x (\alpha(x) \wedge \forall y \leq x \, \neg \beta(y))$.
\item 
$\varphi \preccurlyeq \psi : \equiv \exists x (\alpha(x) \wedge \forall y < x \, \neg \beta(y))$.
\end{itemize}

We apply the witness comparison notation to formulas of the forms $\neg \forall x \varphi(x)$ and $\exists x \alpha(x) \vee \exists x \beta(x)$ by considering $\exists x \neg \varphi(x)$ and $\exists x (\alpha(x) \vee \beta(x))$, respectively. 

\begin{prop}[cf.~Lindstr\"om {\cite[Lemma 1.3]{Lin}}]\label{wc}
For any formulas $\varphi \equiv \exists x \alpha(x)$ and $\psi \equiv \exists x \beta(x)$, the following clauses hold: 
\begin{enumerate}
\item 
$\PA \vdash \varphi \prec \psi \to \varphi \preccurlyeq \psi$.
\item 
$\PA \vdash \varphi \vee \psi \to \varphi \prec \psi \vee \psi \preccurlyeq \varphi$.
\item 
$\PA \vdash \neg (\varphi \prec \psi \wedge \psi \preccurlyeq \varphi)$.
\item 
$\PA \vdash \varphi \wedge \neg \psi \to \varphi \prec \psi$.
\end{enumerate}
\end{prop}

For each $\Gamma \in \{\Sigma_n, \Pi_n \mid n \geq 1\}$,
we introduce the \textit{relativized proof predicate} with respect to $\Gamma$ as follows (See Lindstr\"{o}m \cite{Lin}).
\[
\Prf_{T}^{\Gamma}(x,y) \equiv \exists u \leq y
\bigl(\Gamma(u) \wedge \True_{\Gamma}(u) \wedge \Prf_T(u \dot{\to} x, y) \bigr).
\]
Here, $x \dot{\to} y$ denotes a primitive recursive term corresponding to the primitive function calculating the G\"{o}del number of $\varphi \to \psi$ from those of $\varphi$ and $\psi$.
Note that the formula $\Prf_{T}^{\Gamma}(x,y)$ is $\PA$-equivalent to a $\Gamma$ formula.
We also introduce the relativized proof predicate with respect to $\Sigma_n \land \Pi_n$ ($n \geq 1$) as follows (cf.~H\'{a}jek \cite{Haj79}): 
\begin{align*}
\Prf_{T}^{\Sigma_n \wedge \Pi_n}(x,y) : \equiv \exists u, v \leq y \bigl(\Sigma_n(u) & \wedge \Pi_n(v) \wedge \True_{\Sigma_n}(u) \\ & \wedge \True_{\Pi_n}(v) \wedge \Prf_T(u \dot{\wedge} v \dot{\to} x, y) \bigr).
\end{align*}
Here, $u \dot{\wedge} v$ denotes a primitive recursive term corresponding to the  primitive function calculating the G\"{o}del number of $\varphi \wedge \psi$ from those of $\varphi$ and $\psi$. 
Note that the formula $\Prf_{T}^{\Sigma_n \wedge \Pi_n}(x,y)$ is $\PA$-equivalent to a $\Sigma_{n+1}$ formula.
For each $\Gamma \in \{\Sigma_n, \Pi_n, \Sigma_n \land \Pi_n \mid n \geq 1\}$, we define $\PR_{T}^{\Gamma}(x) : \equiv \exists y \Prf_{T}^{\Gamma}(x,y)$. 

\begin{prop}[cf.~Lindstr\"{o}m \cite{Lin}] \label{ref}
Let $\Gamma \in \{ \Sigma_n, \Pi_n, \Sigma_n \wedge \Pi_n \mid n \geq 1 \}$,
 $\gamma$ be any $\Gamma$ sentence, and $\varphi$ be any sentence. 
 The following clauses hold: 
\begin{enumerate}
\item 
$T \vdash \Prf_{T}^{\Gamma}(\gn{\varphi}, \num{p}) \to \varphi$ for all $p \in \omega$.
\item 
If $T+ \gamma \vdash \varphi$, then $\PA + \gamma \vdash \Prf_{T}^{\Gamma}(\gn{\varphi}, \num{p})$ for some $p \in \omega$.
\end{enumerate}
\end{prop}
Throughout the present paper, we heavily use Propositions \ref{wc} and \ref{ref} without referring to them.

\section{Partially conservative sentences}\label{Sec:Cons}

In this section, we review the known results on partially conservative sentences that are related to the present paper. 
We assume that $\Gamma, \Gamma'$ denote one of $\Delta_n$, $\Sigma_n$, $\Pi_n$, $\Sigma_n \land \Pi_n$, and $\mathcal{B}(\Sigma_n)$ for some $n \geq 1$. 
Also, a formula is called $\mathcal{B}(\Sigma_n)$ iff it is a Boolean combination of $\Sigma_n$ formulas. 

We assume that $\Theta, \Theta'$ denote one of $\Delta_n(\PA)$, $\Sigma_n$, $\Pi_n$, $\Sigma_n \land \Pi_n$, and $\mathcal{B}(\Sigma_n)$ for some $n \geq 1$. 
For each such a class $\Theta$, let $\Theta$ also denote the set of all $\Theta$ sentences. 
Let $\Th(T)$ and $\Th_\Theta(T)$ denote the set of all theorems of $T$ and the set of all $T$-provable $\Theta$ sentences, respectively. 

\subsection{Partially conservative sentences over single theory}

\begin{defn}[$\Gamma$-conservativity]
Let $T$ and $U$ be any theories. 
\begin{itemize}
    \item For $\Gamma \neq \Delta_n$, we say that $T$ is \textit{$\Gamma$-conservative over $U$} iff for any $\Gamma$ sentence $\gamma$, if $T \vdash \gamma$, then $U \vdash \gamma$. 

    \item We say that $T$ is \textit{$\Delta_n$-conservative over $U$} iff for any sentence $\delta$ which is $\Delta_n(\Th(T) \cap \Th(U))$, if $T \vdash \delta$, then $U \vdash \delta$. 
    
    \item We say that a sentence $\varphi$ is \textit{$\Gamma$-conservative over $T$} iff $T + \varphi$ is $\Gamma$-conservative over $T$. 
    \item Let $\Cons(\Gamma, T)$ be the set of all sentences which are $\Gamma$-conservative over $T$. 
\end{itemize}
\end{defn}

Furthermore, Guaspari \cite{Gua} introduced the two notions, hereditary $\Gamma$-conservativity and exact $\Gamma$-conservativity. 

\begin{defn}[Hereditary $\Gamma$-conservativity]
Let $T$ be any theory and $\varphi$ be any sentence. 
\begin{itemize}
    \item We say that $\varphi$ is \textit{hereditarily $\Gamma$-conservative over $T$} iff $\varphi \in \bigcap \{\Cons(\Gamma, U) \mid T \vdash U \vdash \PA\}$. 
    
    \item Let $\HCons(\Gamma, T)$ be the set of all sentences which are hereditarily $\Gamma$-conservative over $T$. 
\end{itemize}
\end{defn}

In the definition of hereditary $\Gamma$-conservativity over $T$, we only refer to the subtheories of $T$ that are extensions of $\PA$. 
In this sense, for the sake of simplicity, we assume that a `subtheory' means a subtheory that includes $\PA$ throughout the present paper. 

\begin{defn}[Exact $\Gamma$-conservativity]\label{EC}
Let $T$ be any theory and $\varphi$ be any sentence. 
\begin{itemize}
    \item We say that $\varphi$ is \textit{exactly $\Gamma$-conservative over $T$} iff $\varphi \in \Cons(\Gamma, T) \setminus \bigcup \{\Cons(\Gamma', T) \mid \Gamma' \nsubseteq \Gamma\}$. 
    
    \item We say that $\varphi$ is \textit{exactly hereditarily $\Gamma$-conservative over $T$} iff $\varphi \in \HCons(\Gamma, T)$ and for each $\Theta$ with $\Theta \nsubseteq \Gamma$, there exists a $\Theta$ sentence $\psi$ such that $\PA + \varphi \vdash \psi$ and $T \nvdash \psi$. 
\end{itemize}
\end{defn}

\begin{rem}\label{Rem_exact}
\begin{itemize}
    \item If a sentence $\varphi$ is exactly $\Gamma$ conservative over $T$, then we have $T \nvdash \varphi$. 
    
    \item The notion of $\Sigma_n \land \Pi_n$-conservativity seems probably not to have been investigated in the literature so far. 
    Also, Guaspari \cite{Gua} introduced the notion of exact $\Gamma$-conservativity only for $\Gamma \in \{\Delta_n, \Sigma_n, \Pi_n \mid n \geq 1\}$.

    \item Definition \ref{EC} indicates that a sentence $\varphi$ is exactly hereditarily $\mathcal{B}(\Sigma_n)$-conservative over $T$ if and only if $\varphi \in \HCons(\mathcal{B}(\Sigma_n), T)$ and there exists a $\Delta_{n+1}(\PA)$ sentence $\psi$ such that $\PA + \varphi \vdash \psi$ and $T \nvdash \psi$. 

    \item The notion of exact $\Gamma$-conservativity depends on which classes $\Gamma'$ ranges over. 
    For example, Definition \ref{EC} indicates that a sentence $\varphi$ is exactly $\Delta_n$-conservative over $T$ if and only if $\varphi \in \Cons(\Delta_n, T) \setminus (\Cons(\Sigma_n, T) \cup \Cons(\Pi_n, T))$. 
    This is precisely the definition of exact $\Delta_n$-conservativity adopted by Guaspari \cite[Definition 2.8]{Gua}.
    However, we will introduce a new notion of $\SP$-conservativity in Section \ref{Sec:SP}, after which we also allow $\Gamma' = \SP$. 
    Then, the notion of exact $\Delta_n$-conservativity will become stronger: a sentence $\varphi$ is exactly $\Delta_n$-conservative over $T$ if and only if $\varphi \in \Cons(\Delta_n, T) \setminus \Cons(\SP, T)$. 
\end{itemize}
\end{rem}

Finally, Guaspari also introduced the notion of essential $\Theta$ sentences. 

\begin{defn}[Essentially $\Theta$ sentences]\label{essentiality}
Let $T$ be any theory. 
We say that a sentence $\varphi$ is called \textit{essentially $\Theta$ over $T$} iff $\varphi \in \Theta$ and for any $\Theta'$ with $\Theta \nsubseteq \Theta'$, there is no $\psi \in \Theta'$ such that $T \vdash \varphi \leftrightarrow \psi$. 
\end{defn}

Guaspari proved the following theorem on the existence of a sentence which is essentially $\Theta$ and exactly hereditarily $\Gamma$-conservative over a single theory. 

\begin{thm}[Guaspari {\cite[Theorem 2.6]{Gua}}]\label{Thm_Gua}
Let $T$ be any theory and $n \geq 1$. 
\begin{enumerate}
    \item For any $\Theta \in \{\Sigma_{m+1}, \Pi_{m} \mid m \geq n\}$, there exists a sentence which is essentially $\Theta$ and exactly hereditarily $\Sigma_n$-conservative over $T$. 

    \item For any $\Theta \in \{\Sigma_{m}, \Pi_{m+1} \mid m \geq n\}$, there exists a sentence which is essentially $\Theta$ and exactly hereditarily $\Pi_n$-conservative over $T$.
\end{enumerate}
\end{thm}

Thus there exists a sentence that is partially conservative over a single theory in a strong sense for the possible combinations $(\Gamma, \Theta)$ if $\Gamma$ and $\Theta$ range over $\{\Sigma_n, \Pi_n \mid n \geq 1\}$. 
However, the situation may become more complicated when $\Gamma$ and $\Theta$ range over more classes.
For example, Guaspari proved the following theorem. 

\begin{thm}[Guaspari {\cite[Corollary 2.9]{Gua}}]\label{Thm_Gua_D}
Let $T$ be any theory and $m \geq n \geq 1$. 
There exists a sentence which is essentially $\Sigma_m \land \Pi_m$ and exactly $\Delta_n$-conservative over $T$. 
\end{thm}

However, it was not clear whether this theorem can be made into a hereditary version, and furthermore, Guaspari wrote that he did not know whether an essentially $\Sigma_n$ sentence which is exactly $\Delta_n$-conservative exists.
Then, Guaspari proposed the following problem. 

\begin{prob}[Guaspari {\cite[p.62, Question (2)]{Gua}}]\label{Gua_Prob}
For any reasonable combination $(\Gamma, \Theta)$ of $\Pi$, $\Sigma$, and $\Delta$, does there exist a sentence which is essentially $\Theta$ and exactly hereditarily $\Gamma$-conservative over $T$?
\end{prob}

At the last part of the present paper, we will answer to this question in a general setting where $\Gamma$ ranges over $\Delta_n$, $\SP$, $\Sigma_n$, $\Pi_n$, $\Sigma_n \land \Pi_n$, and $\mathcal{B}(\Sigma_n)$, and $\Theta$ ranges over $\Delta_n(\PA)$, $\Sigma_n$, $\Pi_n$, $\Sigma_n \land \Pi_n$, and $\mathcal{B}(\Sigma_n)$ for $n \geq 1$ (Theorem \ref{MT}). 
Here, $\SP$-conservativity will be introduced in Section \ref{Sec:SP}. 
Figure \ref{Fig1} shows the implications between the variations of $\Gamma$-conservativity. 

\begin{figure}[ht]
\centering
\begin{tikzpicture}
\node (D+) at (0,1) {$\Delta_{n+1}$-cons.};
\node (B) at (0,0) {$\mathcal{B}(\Sigma_{n})$-cons.};
\node (and) at (0,-1) {$\Sigma_{n} \land \Pi_{n}$-cons.};
\node (S) at (2,-2) {$\Pi_{n}$-cons.};
\node (P) at (-2,-2) {$\Sigma_{n}$-cons.};
\node (SP) at (0,-3) {$\SP$-cons.};
\node (D) at (0,-4) {$\Delta_{n}$-cons.};

\draw [->, double] (D+)--(B);
\draw [->, double] (B)--(and);
\draw [->, double] (and)--(S);
\draw [->, double] (and)--(P);
\draw [->, double] (S)--(SP);
\draw [->, double] (P)--(SP);
\draw [->, double] (SP)--(D);

\end{tikzpicture}
\caption{Implications between the variations of partial conservativity}\label{Fig1}
\end{figure}

\subsection{Simultaneously partially conservative sentences over two theories}

\begin{defn}[Simultaneous non-trivial $\Gamma$-conservativity]
Let $T$ and $U$ be any theories and $\varphi$ be any sentence. 
\begin{itemize}
    \item We say that $\varphi$ is \textit{simultaneously non-trivially $\Gamma$-conservative over $T$ and $U$} iff $\varphi \in \Cons(\Gamma, T) \cap \Cons(\Gamma, U) \setminus (\Th(T) \cup \Th(U))$. 

    \item We say that $\varphi$ is \textit{simultaneously non-trivially hereditarily $\Gamma$-conservative over $T$ and $U$} iff $\varphi \in \HCons(\Gamma, T) \cap \HCons(\Gamma, U) \setminus (\Th(T) \cup \Th(U))$. 
\end{itemize}
\end{defn}

We introduce the following definition. 

\begin{defn}[$\Ben_n$-pairs and $\Benp_n$-pairs]\label{Bennet_pair}
Let $T$ and $U$ be any theories and $n \geq 1$.  
\begin{itemize}
    \item We say that a pair $(T, U)$ of theories is a \textit{$\Ben_n$-pair} iff ($\Th_{\Pi_n}(T) \nsubseteq \Th(U)$ or $T + \Th_{\Pi_n}(U)$ is consistent) and ($\Th_{\Pi_n}(U) \nsubseteq \Th(T)$ or $U + \Th_{\Pi_n}(T)$ is consistent). 

    \item We say that $(T, U)$ is a \textit{$\Benp_n$-pair} iff both $T + \Th_{\Pi_n}(U)$ and $U + \Th_{\Pi_n}(T)$ are consistent. 
\end{itemize}
\end{defn}

Here, `$\Ben$' stands for `Bennet', paying tribute to his work in introducing these conditions. 
Also, `$\Benp$' stands for `consistency' because the condition $\Benp_n$ states the consistency of some theories. 
It is easily shown that a pair $(T, U)$ of theories is a $\Ben_n$-pair if and only if $(T, U)$ is a $\Benp_n$-pair or ($\Th_{\Pi_n}(T) \nsubseteq \Th(U)$ and $\Th_{\Pi_n}(U) \nsubseteq \Th(T)$). 
So, each $\Benp_n$-pair is also a $\Ben_n$-pair. 

\begin{thm}[Bennet \cite{Ben86,Ben}]\label{Thm_Ben}
For any theories $T$ and $U$ and any $n \geq 1$, the following are equivalent: 
\begin{enumerate}
    \item $(T, U)$ is a $\Ben_n$-pair. 
    \item There exists a $\Pi_n$ sentence which is simultaneously non-trivially $\Sigma_n$-conservative over $T$ and $U$, that is,
    \[
        \Pi_n \cap \Cons(\Sigma_n, T) \cap \Cons(\Sigma_n, U) \setminus (\Th(T) \cup \Th(U)) \neq \emptyset.
    \]
    \end{enumerate}
\end{thm}

Bennet's proof of Theorems \ref{Thm_Ben} is based on the following theorem. 

\begin{thm}[Bennet \cite{Ben86,Ben}]\label{Thm_Ben_char}
For any theories $T$ and $U$ and any $n \geq 1$, the following are equivalent: 
\begin{enumerate}
    \item $\Th_{\Pi_n}(T) \nsubseteq \Th(U)$ or $T + \Th_{\Pi_n}(U)$ is consistent. 
    \item $\Pi_n \cap \Cons(\Sigma_n, T) \setminus \Th(U) \neq \emptyset$. 
    \end{enumerate}
\end{thm}

More generally, for finite number of theories, the necessary and sufficient condition for the existence of a $\Pi_n$ sentence which is simultaneously non-trivially $\Sigma_n$-conservative over those theories is given in KOSV \cite[cf.~Theorem 4.6]{KOSV}. 
On the other hand, for the $\Pi_n$-conservativity, the following theorem was proved.  

\begin{thm}[KOSV {\cite[Corollary 5.3]{KOSV}}]\label{Thm_KOSV}
For any finite family of theories $\{T_i \mid i \leq k\}$ and any $n \geq 1$, there exists a $\Sigma_n$ sentence which is simultaneously non-trivially $\Pi_n$-conservative over $T_i$ for all $i \leq k$, that is,
    \[
        \Sigma_n \cap \bigcap_{i \leq k} \bigl(\Cons(\Pi_n, T_i) \setminus \Th(T_i) \bigr) \neq \emptyset.
    \]
\end{thm}

Bennet also studied simultaneous non-trivial hereditary $\Gamma$-conservativity. 

\begin{thm}[Bennet \cite{Ben86,Ben}]\label{Thm_Ben_H}
For any theories $T$ and $U$ and any $n \geq 1$, the following are equivalent: 
\begin{enumerate}
    \item $(T, U)$ is a $\Benp_n$-pair. 
    \item There exists a $\Pi_n$ sentence which is simultaneously non-trivially hereditarily $\Sigma_n$-conservative over $T$ and $U$, that is,
    \[
        \Pi_n \cap \HCons(\Sigma_n, T) \cap \HCons(\Sigma_n, U) \setminus (\Th(T) \cup \Th(U)) \neq \emptyset.
    \]
    \item There exists a $\Sigma_n$ sentence which is simultaneously non-trivially hereditarily $\Pi_n$-conservative over $T$ and $U$, that is,
    \[
        \Sigma_n \cap \HCons(\Pi_n, T) \cap \HCons(\Pi_n, U) \setminus (\Th(T) \cup \Th(U)) \neq \emptyset.
    \]
    \end{enumerate}
\end{thm}

Bennet's proof of Theorem \ref{Thm_Ben_H} uses the following theorem. 

\begin{thm}[Bennet \cite{Ben86,Ben}]\label{Thm_Ben_char_H}
For any theories $T$ and $U$ and any $n \geq 1$, the following are equivalent: 
\begin{enumerate}
    \item $T + \Th_{\Pi_n}(U)$ is consistent. 
    \item $\Sigma_n \cap \HCons(\Pi_n, U) \setminus \Th(T) \neq \emptyset$. 
    \item $\Pi_n \cap \HCons(\Sigma_n, T) \setminus \Th(U) \neq \emptyset$. 
    \end{enumerate}
\end{thm}

Theorems \ref{Thm_Ben_H} and \ref{Thm_Ben_char_H} are further generalized in KOSV \cite[Proposition 4.13 and Corollary 4.14]{KOSV} to families of theories.

\section{$\Sigma_n$-conservativity and $\Pi_n$-conservativity}\label{Sec:Sigma_Pi}

In this section, we refine Bennet's theorems (Theorems \ref{Thm_Ben}, \ref{Thm_Ben_char}, \ref{Thm_Ben_H} and \ref{Thm_Ben_char_H}) on $\Sigma_n$-conservativity and $\Pi_n$-conservativity. 
First of all, we introduce the notion of exact $\Gamma$-conservativity over two theories. 

\begin{defn}[Simultaneous exact $\Gamma$-conservativity]
Let $T$ and $U$ be any theories and $\varphi$ be any sentence. 
\begin{itemize}
    \item We say that $\varphi$ is \textit{simultaneously exactly $\Gamma$-conservative over $T$ and $U$} iff $\varphi \in \Cons(\Gamma, T) \cap \Cons(\Gamma, U)$ and for each $\Theta$ with $\Theta \nsubseteq \Gamma$, there exists a $\Theta$ sentence $\psi$ such that $T + \varphi \vdash \psi$, $U + \varphi \vdash \psi$, $T \nvdash \psi$, and $U \nvdash \psi$. 

    \item We say that $\varphi$ is \textit{simultaneously exactly hereditarily $\Gamma$-conservative over $T$ and $U$} iff $\varphi \in \HCons(\Gamma, T) \cap \HCons(\Gamma, U)$ and for each $\Theta$ with $\Theta \nsubseteq \Gamma$, there exists a $\Theta$ sentence $\psi$ such that $\PA + \varphi \vdash \psi$, $T \nvdash \psi$, and $U \nvdash \psi$.
\end{itemize}
\end{defn}

\begin{rem}
    We adopt a strong one of the possible variations as the definition of simultaneous `exact' hereditary $\Gamma$-conservativity over $T$ and $U$. 
    Other candidates are as follows: 
    \begin{enumerate}
        \item $\varphi \in \HCons(\Gamma, T) \cap \HCons(\Gamma, U)$ and for each $\Theta$ with $\Theta \nsubseteq \Gamma$, there exist a subtheory $V$ of $T$ or $U$ and a $\Theta$ sentence $\psi$ such that $V + \varphi \vdash \psi$ and $V \nvdash \psi$.

        \item For every subtheory $V$ of $T$ or $U$, $\varphi$ is simultaneously exactly $\Gamma$-conservative over $V$.
    \end{enumerate}
The second one should rather be called simultaneous hereditary exact $\Gamma$-conservativity. 
Both of them are natural, but we leave the analysis of these notions to future studies. 
\end{rem}

As described in the introduction, from now on, we will analyze the status of the following four notions for each $\Gamma$. 
Our results will be summarized in Section \ref{Sec:Summary}. 

\begin{itemize}
    \item Simultaneous non-trivial hereditary $\Gamma$-conservativity. 
    \item Simultaneous exact hereditary $\Gamma$-conservativity. 
    \item Simultaneous non-trivial $\Gamma$-conservativity. 
    \item Simultaneous exact $\Gamma$-conservativity. 
\end{itemize}

\subsection{Simultaneous hereditary conservativity}

First of all, we prove the following useful lemma. 

\begin{lem}\label{Lem_useful}
For any theories $T$ and $U$ and any class $\Gamma$, if $\HCons(\Gamma, U) \setminus \Th(T) \neq \emptyset$, then $T + \Th_\Gamma(U)$ is consistent. 
\end{lem}
\begin{proof}
Let $\varphi$ be a sentence with $\varphi \in \HCons(\Gamma, U) \setminus \Th(T)$. 
Suppose, towards a contradiction, that $T + \Th_{\Gamma}(U)$ is inconsistent. 
Then, we find a $\Gamma$ sentence $\gamma$ such that $U \vdash \gamma$ and $T \vdash \neg \gamma$. 
Here, $\PA + \gamma \lor \neg \varphi$ is a subtheory of $U$ and we have $\PA + \gamma \lor \neg \varphi + \varphi \vdash \gamma$. 
Since $\varphi \in \HCons(\Gamma, U)$, we obtain $\PA + \gamma \lor \neg \varphi \vdash \gamma$. 
So, $\PA + \neg \varphi \vdash \gamma$, and hence $\PA + \neg \gamma \vdash \varphi$. 
Thus, $T \vdash \varphi$, a contradiction. 
We have shown that $T + \Th_{\Gamma}(U)$ is consistent. 
\end{proof}

For $\Gamma \in \{\Sigma_n, \Pi_n \mid n \geq 1\}$, we show that simultaneous non-trivial hereditary $\Gamma$-conservativity and simultaneous exact hereditary $\Gamma$-conservativity are equivalent by refining Theorem \ref{Thm_Ben_char_H} as follows. 

\begin{thm}\label{Thm_char_H_refine}
For any theories $T$ and $U$ and any $n \geq 1$, the following are equivalent: 
\begin{enumerate}
    \item $T + \Th_{\Pi_n}(U)$ is consistent. 
    \item $\Sigma_n \cap \HCons(\Pi_n, T) \cap \HCons(\Pi_n, U) \setminus \Th(T) \neq \emptyset$. 
    \item $\HCons(\Pi_n, U) \setminus \Th(T) \neq \emptyset$. 
    \item $\Pi_n \cap \HCons(\Sigma_n, T) \cap \HCons(\Sigma_n, U) \setminus \Th(U) \neq \emptyset$. 
    \item $\HCons(\Sigma_n, T) \setminus \Th(U) \neq \emptyset$. 
    \item There exists a sentence $\varphi \in \HCons(\Delta_n, T)$ such that $\PA + \varphi$ is not $\Pi_n$-conservative over $U$. 
    \end{enumerate}
\end{thm}
\begin{proof}
$(1 \Rightarrow 2)$: By the Fixed Point Theorem (cf.~Lindstr\"om \cite{Lin}), we find a $\Sigma_n$ sentence $\varphi$ satisfying the following equivalence:
\[
    \PA \vdash \varphi \leftrightarrow \bigl(\PR_T^{\Sigma_n}(\gn{\neg \varphi}) \preccurlyeq \PR_T(\gn{\varphi}) \bigr) \lor \bigl(\PR_U^{\Sigma_n}(\gn{\neg \varphi}) \preccurlyeq \PR_T(\gn{\varphi}) \bigr).
\]
We show that $\varphi$ is in $\HCons(\Pi_n, T) \cap \HCons(\Pi_n, U) \setminus \Th(T)$. 

If $T \vdash \varphi$, then we find a natural number $p$ such that $\PA \vdash \Prf_T(\gn{\varphi}, \overline{p})$. 
Since $T \vdash \forall z \leq \overline{p}\, \neg \Prf_T^{\Sigma_n}(\gn{\neg \varphi}, z)$, we obtain $T \vdash \PR_T(\gn{\varphi}) \prec \PR_T^{\Sigma_n}(\gn{\neg \varphi})$. 
Then, $T \vdash \neg \bigl(\PR_T^{\Sigma_n}(\gn{\neg \varphi}) \preccurlyeq \PR_T(\gn{\varphi}) \bigr)$. 
By combining this with $T \vdash \varphi$, we have $T \vdash \PR_U^{\Sigma_n}(\gn{\neg \varphi}) \preccurlyeq \PR_T(\gn{\varphi})$. 
On the other hand, since $U + \varphi \vdash \forall z \leq \overline{p}\, \neg \Prf_U^{\Sigma_n}(\gn{\neg \varphi}, z)$, we obtain $U \vdash \neg \bigl(\PR_U^{\Sigma_n}(\gn{\neg \varphi}) \preccurlyeq \PR_T(\gn{\varphi}) \bigr)$ as above.
This contradicts the consistency of $T + \Th_{\Pi_n}(U)$. 
Therefore, $T \nvdash \varphi$. 

We prove only $\varphi \in \HCons(\Pi_n, T)$ and $\varphi \in \HCons(\Pi_n, U)$ is proved similarly. 
Let $V$ be any subtheory of $T$ and $\pi$ be any $\Pi_n$ sentence with $V + \varphi \vdash \pi$. 
Since $T + \neg \pi \vdash \neg \varphi$, we have $\PA + \neg \pi \vdash \Prf_T^{\Sigma_n}(\gn{\neg \varphi}, \overline{p})$ for some $p \in \omega$. 
Since $T \nvdash \varphi$, we have $\PA + \neg \pi \vdash \PR_T^{\Sigma_n}(\gn{\neg \varphi}) \preccurlyeq \PR_T(\gn{\varphi})$, and hence $\PA + \neg \pi \vdash \varphi$. 
Since $V + \neg \pi$ proves $\pi$ and $\neg \pi$, the theory is inconsistent. 
This means that $V \vdash \pi$. 

\medskip

$(2 \Rightarrow 3)$: Obvious. 

\medskip

$(3 \Rightarrow 1)$: This implication is a direct consequence of Lemma \ref{Lem_useful}. 

\medskip

$(1 \Rightarrow 4)$: 
By the Fixed Point Theorem, we find a $\Pi_n$ sentence $\varphi$ satisfying the following equivalence:
\[
    \PA \vdash \varphi \leftrightarrow \neg \bigl(\PR_U(\gn{\varphi}) \preccurlyeq \PR_T^{\Pi_n}(\gn{\neg \varphi}) \bigr) \lor \neg \bigl(\PR_U(\gn{\varphi}) \preccurlyeq \PR_U^{\Pi_n}(\gn{\neg \varphi}) \bigr).
\]
Then, as above, it is shown that $\varphi \in \HCons(\Sigma_n, T) \cap \HCons(\Sigma_n, U) \setminus \Th(U)$. 

\medskip

$(4 \Rightarrow 5)$: Obvious. 

\medskip

$(5 \Rightarrow 1)$: Suppose $\HCons(\Sigma_n, T) \setminus \Th(U) \neq \emptyset$. 
The consistency of $U + \Th_{\Sigma_n}(T)$ follows from Lemma \ref{Lem_useful}. 
So, $T + \Th_{\Pi_n}(U)$ is also consistent. 

\medskip

$(4 \Rightarrow 6)$: Obvious. 

\medskip

$(6 \Rightarrow 1)$: Let $\varphi$ be a sentence such that $\varphi \in \HCons(\Delta_n, T)$ and $\pi \in \Pi_n$ be such that $\PA + \varphi \vdash \pi$ and $U \nvdash \pi$.
Suppose, towards a contradiction, that $T + \Th_{\Pi_n}(U)$ is inconsistent.
Then, there exists a $\Pi_n$ sentence $\psi$ such that $U \vdash \psi$ and $T \vdash \neg \psi$.
Since $\PA + \neg \psi \vee \neg \pi   \vdash \neg \psi \preccurlyeq \neg \pi \leftrightarrow \neg (\neg \pi \prec \neg \psi)$, the $\Sigma_n$ sentence $\neg \psi \preccurlyeq \neg \pi$ is $\Delta_n(\PA + \neg \psi \vee \neg \pi)$. 
Since $\PA + \varphi \vdash \pi$, we obtain $\PA + \neg \psi \vee \neg \pi + \varphi \vdash \neg \psi \wedge \pi$, which implies $\PA + \neg \psi \vee \neg \pi + \varphi \vdash \neg \psi \preccurlyeq \neg \pi$. Since $T \vdash \PA + \neg \psi \vee \neg \pi$ and $\varphi \in \HCons(\Delta_n, T)$, we obtain $\PA + \neg \psi \vee \neg \pi \vdash \neg \psi \preccurlyeq \neg \pi$.
Hence, $\PA + \neg \pi \vdash \neg \psi \preccurlyeq \neg \pi$ holds and we obtain $\PA + \neg \pi \vdash \neg \psi$.
Since $U \vdash \psi$, it follows that $U \vdash \pi$. This contradicts $U \nvdash \pi$.
\end{proof}

The following theorem is a refinement of the equivalence $(1 \Leftrightarrow 2)$ of Theorem \ref{Thm_Ben_H}. 
In fact, the equivalence $(1 \Leftrightarrow 2)$ of the following Theorem \ref{Thm_Pi_H} is exactly the equivalence $(1 \Leftrightarrow 2)$ of Theorem \ref{Thm_Ben_H}. 
We present a simple proof of the implication $(1 \Rightarrow 2)$ using Theorem \ref{Thm_char_H_refine} instead of Theorem \ref{Thm_Ben_char_H}. 

\begin{thm}\label{Thm_Pi_H}
For any theories $T$ and $U$ and any $n \geq 1$, the following are equivalent: 
\begin{enumerate}
    \item $(T, U)$ is a $\Benp_n$-pair. 
    \item There exists a $\Sigma_n$ sentence which is simultaneously non-trivially hereditarily $\Pi_n$-conservative over $T$ and $U$, that is, 
    \[
        \Sigma_n \cap \HCons(\Pi_n, T) \cap \HCons(\Pi_n, U) \setminus (\Th(T) \cup \Th(U)) \neq \emptyset. 
    \]
    \item There exists a sentence which is simultaneously exactly hereditarily $\Pi_n$-conservative over $T$ and $U$, that is, there exist sentences $\varphi$ and $\psi$ such that $\varphi \in \HCons(\Pi_n, T) \cap \HCons(\Pi_n, U)$, $\psi \in \Sigma_n$, $\PA + \varphi \vdash \psi$, $T \nvdash \psi$, and $U \nvdash \psi$. 
    \item There exists a sentence which is simultaneously non-trivially hereditarily $\Pi_n$-conservative over $T$ and $U$, that is, 
      \[
        \HCons(\Pi_n, T) \cap \HCons(\Pi_n, U) \setminus (\Th(T) \cup \Th(U)) \neq \emptyset. 
    \]
\end{enumerate}
\end{thm}
\begin{proof}
$(1 \Rightarrow 2)$: 
Suppose $(T, U)$ is a $\Benp_n$-pair. 
Since $T + \Th_{\Pi_n}(U)$ is consistent, by Theorem \ref{Thm_char_H_refine}, we find $\varphi \in \Sigma_n \cap \HCons(\Pi_n, T) \cap \HCons(\Pi_n, U) \setminus \Th(T)$. 
Also, we find $\psi \in \Sigma_n \cap \HCons(\Pi_n, T) \cap \HCons(\Pi_n, U) \setminus \Th(U)$ because $U + \Th_{\Pi_n}(T)$ is consistent. 
We show that the $\Sigma_n$ sentence $\varphi \land \psi$ witnesses the simultaneous non-trivial hereditary $\Pi_n$-conservativity. 
We have $T \nvdash \varphi \land \psi$ and $U \nvdash \varphi \land \psi$ obviously. 

Let $V$ be any subtheory of $T$ or $U$ and let $\pi \in \Pi_n$ be such that $V + \varphi \land \psi \vdash \pi$. 
Since $\psi \to \pi$ is $\Pi_n$, by the choice of $\varphi$, we obtain $V \vdash \psi \to \pi$. 
Also by the choice of $\psi$, we have $V \vdash \pi$. 

\medskip

$(2 \Rightarrow 3)$: 
This is because every $\Sigma_n$ sentence which is simultaneously non-trivially hereditarily $\Pi_n$-conservative is simultaneously exactly hereditarily $\Pi_n$-conservative. 

\medskip

$(3 \Rightarrow 4)$: Obvious. 

\medskip

$(4 \Rightarrow 1)$: 
  This follows immediately from Theorem \ref{Thm_char_H_refine}. 
\end{proof}

Similarly, the following refinement of the equivalence $(1 \Leftrightarrow 3)$ of Theorem \ref{Thm_Ben_H} follows from Theorem \ref{Thm_char_H_refine}. 

\begin{thm}\label{Thm_Sigma_H}
For any theories $T$ and $U$ and any $n \geq 1$, the following are equivalent: 
\begin{enumerate}
    \item $(T, U)$ is a $\Benp_n$-pair. 
    \item There exists a $\Pi_n$ sentence which is simultaneously non-trivially hereditarily $\Sigma_n$-conservative over $T$ and $U$, that is, 
    \[
        \Pi_n \cap \HCons(\Sigma_n, T) \cap \HCons(\Sigma_n, U) \setminus (\Th(T) \cup \Th(U)) \neq \emptyset.
    \]
    \item There exists a sentence which is simultaneously exactly hereditarily $\Sigma_n$-conservative over $T$ and $U$, that is, there exist sentences $\varphi$ and $\psi$ such that $\varphi \in \HCons(\Sigma_n, T) \cap \HCons(\Sigma_n, U)$, $\psi \in \Pi_n$, $\PA + \varphi \vdash \psi$, $T \nvdash \psi$, and $U \nvdash \psi$.  
    \item There exists a sentence which is simultaneously non-trivially hereditarily $\Sigma_n$-conservative over $T$ and $U$, that is, 
    \[
        \HCons(\Sigma_n, T) \cap \HCons(\Sigma_n, U) \setminus (\Th(T) \cup \Th(U)) \neq \emptyset.
    \]
\end{enumerate}
\end{thm}

\subsection{Simultaneous conservativity}

By Theorem \ref{Thm_KOSV}, for any theories $T$ and $U$, there always exists a $\Sigma_n$ sentence which is simultaneously exactly $\Pi_n$-conservative over $T$ and $U$. 
So, here we only examine simultaneous $\Sigma_n$-conservativity. 
It follows from Theorem \ref{Thm_and} which will be proved in Section \ref{Sec:and} that for any theories $T$ and $U$, there always exists a $\mathcal{B}(\Sigma_n)$ sentence which is simultaneously non-trivially $\Sigma_n$-conservative over $T$ and $U$. 
Hence, we discus the following remaining ones: 
\begin{itemize}
    \item $\Pi_n$ or $\Sigma_n \land \Pi_n$ sentences which are simultaneously non-trivially $\Sigma_n$-conservative, 
    \item simultaneously exactly $\Sigma_n$-conservative sentences.
\end{itemize}

We refine Theorem \ref{Thm_Ben_char} as follows. 

\begin{thm}\label{Thm_char_refine}
For any theories $T$ and $U$ and any $n \geq 1$, the following are equivalent: 
\begin{enumerate}
    \item $\Th_{\Pi_n}(T) \nsubseteq \Th(U)$ or $T + \Th_{\Pi_n}(U)$ is consistent. 
    \item $\Pi_n \cap \Cons(\Sigma_n, T) \cap \Cons(\Sigma_n, U) \setminus \Th(U) \neq \emptyset$. 
    \item There exists a sentence $\varphi \in \Cons(\Delta_n, T)$ such that $T + \varphi$ is not $\Pi_n$-conservative over $U$. 
    \end{enumerate}
\end{thm}
\begin{proof}
$(1 \Rightarrow 2)$: Suppose $\Th_{\Pi_n}(T) \nsubseteq \Th(U)$ or $T + \Th_{\Pi_n}(U)$ is consistent. 
If $T + \Th_{\Pi_n}(U)$ is consistent, then by Theorem \ref{Thm_char_H_refine}, we are done. 
So, we may assume that $\Th_{\Pi_n}(T) \nsubseteq \Th(U)$, that is, $T \vdash \chi$ and $U \nvdash \chi$ for some $\Pi_n$ sentence $\chi$. 
By the Fixed Point Theorem, we find a $\Pi_n$ sentence $\varphi$ satisfying the following equivalence: 
\[
    \PA \vdash \varphi \leftrightarrow \chi \lor \neg \bigl(\PR_U(\gn{\varphi}) \preccurlyeq \PR_U^{\Pi_n}(\gn{\neg \varphi}) \bigr). 
\]
Since $T \vdash \varphi$, we have $\varphi \in \Cons(\Sigma_n, T)$. 

If $U \vdash \varphi$, then we have $\PA \vdash \Prf_U(\gn{\varphi}, \overline{p})$ for some $p \in \omega$. 
Since $U \vdash \forall y < \overline{p}\, \neg \Prf_U(\gn{\neg \varphi}, y)$, we obtain $U \vdash \PR_U(\gn{\varphi}) \preccurlyeq \PR_U^{\Pi_n}(\gn{\neg \varphi})$. 
By combining this with $U \vdash \varphi$, we obtain $U \vdash \chi$, a contradiction. 
Therefore, $U \nvdash \varphi$. 

At last, we show $\varphi \in \Cons(\Sigma_n, U)$. 
Let $\sigma \in \Sigma_n$ be such that $U + \varphi \vdash \sigma$. 
Since $U + \neg \sigma \vdash \neg \varphi$, we have $\PA + \neg \sigma \vdash \Prf_U^{\Pi_n}(\gn{\neg \varphi}, \overline{p})$ for some $p \in \omega$. 
Then we have $\PA + \neg \sigma \vdash \PR_U^{\Pi_n}(\gn{\neg \varphi}) \prec \PR_U(\gn{\varphi})$ because $U \nvdash \varphi$. 
Since $\PA + \neg \sigma \vdash \neg \bigl(\PR_U(\gn{\varphi}) \preccurlyeq \PR_U^{\Pi_n}(\gn{\neg \varphi}) \bigr)$, we obtain $\PA + \neg \sigma \vdash \varphi$. 
By combining this with $U + \varphi \vdash \sigma$, we obtain that $U + \neg \sigma$ is inconsistent. 
Hence, $U \vdash \sigma$.

\medskip

$(2 \Rightarrow 3)$: Trivial.

\medskip

$(3 \Rightarrow 1)$: 
Let $\varphi$ be a sentence such that $\varphi \in \Cons(\Delta_n, T)$ and $\pi \in \Pi_n$ be such that $T + \varphi \vdash \pi$ and $U \nvdash \pi$.  
Suppose, towards a contradiction, that $\Th_{\Pi_n}(T) \subseteq \Th(U)$ and $T + \Th_{\Pi_n}(U)$ is inconsistent. 
Then, we find $\sigma \in \Sigma_n$ such that $T \vdash \sigma$ and $U \vdash \neg \sigma$. 
Since $T \vdash \sigma$, we have $T \vdash \sigma \preccurlyeq \neg \pi \leftrightarrow \neg (\neg \pi \prec \sigma)$, and hence, $\sigma \preccurlyeq \neg \pi$ is $\Delta_n(T)$. 
We have $T + \varphi \vdash \sigma \preccurlyeq \neg \pi$ because $T + \varphi \vdash \pi$. 
Hence $T \vdash \sigma \preccurlyeq \neg \pi$ and $T \vdash \neg (\neg \pi \prec \sigma)$ because $\varphi \in \Cons(\Delta_n, T)$. 
Since $\Th_{\Pi_n}(T) \subseteq \Th(U)$, we get $U \vdash \neg (\neg \pi \prec \sigma)$. 
It follows $U + \neg \pi \vdash \sigma \preccurlyeq \neg \pi$ from $\PA + \neg \pi \vdash (\sigma \preccurlyeq \neg \pi) \lor (\neg \pi \prec \sigma)$. 
Then, $U + \neg \pi \vdash \sigma$. 
Since $U \vdash \neg \sigma$, we obtain $U \vdash \pi$, a contradiction. 
Therefore, we conclude that $\Th_{\Pi_n}(T) \nsubseteq \Th(U)$ or $T + \Th_{\Pi_n}(U)$ is consistent. 
\end{proof}

As a consequence of Theorem \ref{Thm_char_refine}, we obtain the following refinement of Theorem \ref{Thm_Ben}. 

\begin{thm}\label{Thm_Sigma}
For any theories $T$ and $U$ and any $n \geq 1$, the following are equivalent: 
\begin{enumerate}
    \item $(T, U)$ is a $\Ben_n$-pair. 
    \item There exists a $\Pi_n$ sentence which is simultaneously non-trivially $\Sigma_n$-conservative over $T$ and $U$, that is, 
    \[
        \Pi_n \cap \Cons(\Sigma_n, T) \cap \Cons(\Sigma_n, U) \setminus (\Th(T) \cup \Th(U)) \neq \emptyset.
    \]
    \item There exists a $\Sigma_n \land \Pi_n$ sentence which is simultaneously non-trivially $\Sigma_n$-conservative over $T$ and $U$, that is, 
    \[
        (\Sigma_n \land \Pi_n) \cap \Cons(\Sigma_n, T) \cap \Cons(\Sigma_n, U) \setminus (\Th(T) \cup \Th(U)) \neq \emptyset.
    \]
    \item There exists a sentence which is simultaneously exactly $\Sigma_n$-conservative over $T$ and $U$, that is, there exist sentences $\varphi$ and $\psi$ such that $\varphi \in \Cons(\Sigma_n, T) \cap \Cons(\Sigma_n, U)$, $\psi \in \Pi_n$, $T + \varphi \vdash \psi$, $U + \varphi \vdash \psi$, $T \nvdash \psi$, and $U \nvdash \psi$. 
\end{enumerate}
\end{thm}
\begin{proof}
$(1 \Rightarrow 2)$: This is immediate from Theorem \ref{Thm_char_refine} as in the proof of $(1 \Rightarrow 2)$ of Theorem \ref{Thm_Pi_H}. 

\medskip

$(2 \Rightarrow 3)$: This is trivial because every $\Pi_n$-sentence $\pi$ is $\PA$-provably equivalent to the $\Sigma_n \land \Pi_n$ sentence $0=0 \land \pi$. 

\medskip

$(3 \Rightarrow 2)$: Let $\sigma \in \Sigma_n$ and $\pi \in \Pi_n$ be such that $\sigma \land \pi$ is simultaneously non-trivially $\Sigma_n$-conservative over $T$ and $U$. 
Since $T + \sigma \land \pi \vdash \sigma$, we have $T \vdash \sigma$. 
Hence $\pi$ is simultaneously non-trivially $\Sigma_n$-conservative over $T$ and $U$. 

\medskip

$(2 \Rightarrow 4)$: This is because every $\Pi_n$ sentence which is simultaneously non-trivially $\Sigma_n$-conservative is simultaneously exactly $\Sigma_n$-conservative. 

\medskip

$(4 \Rightarrow 1)$: This implication follows from $(3 \Rightarrow 1)$ of Theorem \ref{Thm_char_refine}. 
\end{proof}

\section{$\SP$-conservativity}\label{Sec:SP}

In this section, we newly introduce the following notion of $\SP$-conservativity. 

\begin{defn}[$\SP$-conservativity]
Let $T$ and $U$ be any theories. 
\begin{enumerate}
    \item We say that $T$ is \textit{$\SP$-conservative over $U$} iff for any $\Sigma_n$ sentence $\sigma$ and $\Pi_n$ sentence $\pi$, if $T \vdash \sigma \land \pi$, then $U \vdash \sigma \lor \pi$. 
    \item We say that a sentence $\varphi$ is \textit{$\SP$-conservative over $T$} iff $T + \varphi$ is $\SP$-conservative over $T$. 
\end{enumerate}
\end{defn}

First, we show some basic properties and applications of the notion. 
Secondly, we study simultaneous $\SP$-conservativity over two theories. 

\subsection{Basic properties and applications}

The following proposition shows that $\SP$-conservativity is positioned as drawn in Figure \ref{Fig1}. 

\begin{prop}\label{Prop_SP_1}
Let $T$ and $U$ be any theories.
\begin{enumerate}
    \item If $T$ is $\Sigma_n$-conservative over $U$, then $T$ is $\SP$-conservative over $U$. 

    \item If $T$ is $\Pi_n$-conservative over $U$, then $T$ is $\SP$-conservative over $U$. 

    \item If $T$ is $\SP$-conservative over $U$, then $T$ is $\Delta_n$-conservative over $U$. 
\end{enumerate}
\end{prop}
\begin{proof}
1 and 2 are trivial. 

3. Suppose that $T$ is $\SP$-conservative over $U$. 
Let $\varphi$ be a $\Delta_n(\Th(T) \cap \Th(U))$ sentence such that $T \vdash \varphi$. 
We find $\sigma \in \Sigma_n$ and $\pi \in \Pi_n$ such that $\Th(T) \cap \Th(U) \vdash \varphi \leftrightarrow \sigma$ and $\Th(T) \cap \Th(U) \vdash \varphi \leftrightarrow \pi$. 
Then, $T \vdash \sigma \land \pi$. 
By $\SP$-conservativity, we have $U \vdash \sigma \lor \pi$. 
We conclude $U \vdash \varphi$. 
\end{proof}

So, from now on, the notion of exact $\Delta_n$-conservativity is changed as mentioned in Remark \ref{Rem_exact}. 
Here, we explicitly define the notions of exact $\SP$-conservativity and exact $\Delta_n$-conservativity. 

\begin{defn}[Exact $\SP$-conservativity and exact $\Delta_n$-conservativity]
Let $T$ be any theory and $\varphi$ be any sentence. 
\begin{itemize}
    \item We say that $\varphi$ is \textit{exactly $\SP$-conservative over $T$} iff $\varphi \in \Cons(\SP, T) \setminus (\Cons(\Sigma_n, T) \cup \Cons(\Pi_n, T))$. 

    \item We say that $\varphi$ is \textit{exactly $\Delta_n$-conservative over $T$} iff $\varphi \in \Cons(\Delta_n, T) \setminus \Cons(\SP, T)$. 

    \item The notions such as simultaneous exact hereditary $\SP$-conservativity are also defined similarly. 
\end{itemize}
\end{defn}

Thus, Guaspari's theorem (Theorem \ref{Thm_Gua_D}) no longer shows the existence of an exactly $\Delta_n$-conservative sentence in our sense.
Our stronger version of exact $\Delta_n$-conservativity will be studied in the next section. 

\begin{prop}\label{Prop_SP_2}
Let $T$ be any theory and $n \geq 1$. 
\begin{enumerate}
    \item If $\varphi$ is a $\Sigma_n$ sentence which is $\SP$-conservative over $T$, then $\varphi$ is $\Pi_n$-conservative over $T$. 
    So, there is no $\Sigma_n$ sentence which is exactly $\SP$-conservative over $T$.

    \item If $\varphi$ is a $\Pi_n$ sentence which is $\SP$-conservative over $T$, then $\varphi$ is $\Sigma_n$-conservative over $T$. 
    So, there is no $\Pi_n$ sentence which is exactly $\SP$-conservative over $T$.
\end{enumerate}
\end{prop}
\begin{proof}
1. Suppose that $\varphi$ is a $\Sigma_n$ sentence which is $\SP$-conservative over $T$. 
Let $\pi \in \Pi_n$ be such that $T + \varphi \vdash \pi$. 
Since $T + \varphi \vdash \varphi \land \pi$, we have $T \vdash \varphi \lor \pi$. 
By combining this with $T + \varphi \vdash \pi$, we conclude $T \vdash \pi$. 

2. This is proved in the similar way as 1. 
\end{proof}

We show that the notion of $\SP$-conservativity is useful to study that of hereditary double $\Pi_n$-conservativity.

\begin{defn}[Double $\Pi_n$-conservativity]
Let $T$ be any theory. 
\begin{enumerate}
    \item We say that a sentence $\varphi$ is \textit{doubly $\Pi_n$-conservative} over $T$ iff $\varphi$ and $\neg \varphi$ are $\Pi_n$-conservative and $\Sigma_n$-conservative over $T$, respectively. 

    \item Let $\DCons(\Pi_n, T)$ be the set of all sentences which are doubly $\Pi_n$-conservative over $T$. 
    
    \item We say that a sentence $\varphi$ is \textit{hereditarily doubly $\Pi_n$-conservative} over $T$ iff $\varphi \in \bigcap \{\DCons(\Pi_n, U) \mid T \vdash U \vdash \PA\}$. 
    
    \item Let $\HDCons(\Pi_n, T)$ be the set of all sentences which are hereditarily doubly $\Pi_n$-conservative over $T$. 
\end{enumerate}
\end{defn}

The term `doubly $\Gamma$-conservative' is due to H\'ajek \cite{Haj87}.

\begin{thm}[Solovay (cf.~{\cite[Theorem 2.7]{Gua}}, {\cite[Theorem 3.4.(3)]{Haj87}}\label{Solovay} and {\cite[Theorem 5.3]{Lin}}]
For any theory $T$, we have $\Sigma_n \cap \DCons(\Pi_n, T)\neq \emptyset$. 
\end{thm}

On the other hand, it is known that $\Sigma_n \cap \HDCons(\Pi_n, T)$ is not always non-empty (cf.~\cite[Exercise 5.5.(c)]{Lin}). 
We prove the following theorem expressing this fact more precisely. 
In this study, the notion of $\SP$-conservativity plays a key role. 

\begin{thm}\label{HD1}
    For any theory $T$, the following are equivalent:
    \begin{enumerate}
        \item $\Sigma_n \cap \HDCons(\Pi_n, T) \neq \emptyset$. 
        \item $\Th_{\Sigma_n}(T) \subseteq \HCons(\Pi_n, T)$. 
        \item $\Th_{\Pi_n}(T) \subseteq \HCons(\Sigma_n, T)$. 
        \item $T$ is $\SP$-conservative over $\PA$. 
    \end{enumerate}
\end{thm}
\begin{proof}
$(1 \Rightarrow 2)$: 
Let $\varphi$ be a $\Sigma_n$ sentence such that $\varphi \in \HCons(\Pi_n, T)$ and $\neg \varphi \in \HCons(\Sigma_n, T)$. 
Let $\sigma \in \Sigma_n$ be any theorem of $T$. 
Let $V$ be any subtheory of $T$ and $\pi \in \Pi_n$ be such that $V + \sigma \vdash \pi$. 
Then, we have $T \vdash \sigma \land \pi$. 
Since $\PA + \sigma \lor \varphi$ is a subtheory of $T$ and $\PA + \sigma \lor \varphi + \neg \varphi \vdash \sigma$, we obtain $\PA + \sigma \lor \varphi \vdash \sigma$. 
So, we have $\PA + \varphi \vdash \sigma$. 
Also, we get $\PA + \neg \varphi \vdash \pi$ because $\PA + \pi \lor \neg \varphi$ is a subtheory of $T$ and $\PA + \pi \lor \neg \varphi + \varphi \vdash \pi$. 
Therefore, $\PA \vdash \sigma \lor \pi$. 
By combining this with $V + \sigma \vdash \pi$, we conclude $V \vdash \pi$.

\medskip

$(2 \Rightarrow 3)$: Suppose $\Th_{\Sigma_n}(T) \subseteq \HCons(\Pi_n, T)$. 
Let $\pi \in \Pi_n$ be any theorem of $T$. 
Let $V$ be any subtheory of $T$ and $\sigma \in \Sigma_n$ be such that $V + \pi \vdash \sigma$. 
Then, we have $T \vdash \sigma \land \pi$. 
By the supposition, we get $\sigma \in \HCons(\Pi_n, T)$. 
Since $\PA + \pi \lor \neg \sigma$ is a subtheory of $T$ and $\PA + \pi \lor \neg \sigma + \sigma \vdash \pi$, we obtain $\PA + \pi \lor \neg \sigma \vdash \pi$, and hence $\PA + \neg \sigma \vdash \pi$. 
By combining this with $V + \pi \vdash \sigma$, we have that $V + \neg \sigma$ is inconsistent. 
Therefore, $V \vdash \sigma$. 

\medskip

$(3 \Rightarrow 4)$: Suppose $\Th_{\Pi_n}(T) \subseteq \HCons(\Sigma_n, T)$. 
Let $\sigma \in \Sigma_n$ and $\pi \in \Pi_n$ be such that $T \vdash \sigma \land \pi$. 
By the supposition, $\pi \in \HCons(\Sigma_n, T)$. 
Since $\PA + \sigma \lor \neg \pi$ is a subtheory of $T$ and $\PA + \sigma \lor \neg \pi + \pi \vdash \sigma$, we obtain $\PA + \neg \pi \vdash \sigma$. 
That is, $\PA \vdash \sigma \lor \pi$. 

\medskip

$(4 \Rightarrow 1)$: Suppose that $T$ is $\SP$-conservative over $\PA$. 
By the Fixed Point Theorem, we find a $\Sigma_n$ sentence $\varphi$ satisfying the following equivalence: 
\[
    \PA \vdash \varphi \leftrightarrow \PR_T^{\Sigma_n}(\gn{\neg \varphi}) \preccurlyeq \PR_T^{\Pi_n}(\gn{\varphi}).
\]
Let $V$ be any subtheory of $T$ and we show $\varphi \in \DCons(\Pi_n, V)$. 

Let $\pi$ be any $\Pi_n$ sentence such that $V + \varphi \vdash \pi$. 
Since $T + \neg \pi \vdash \neg \varphi$, we find some $p \in \omega$ such that
\begin{equation}\label{eq1}
    \PA + \neg \pi \vdash \Prf_T^{\Sigma_n}(\gn{\neg \varphi}, \overline{p}).
\end{equation}
By combining this with $T + \neg \varphi \vdash \forall y < \overline{p}\, \neg \Prf_T^{\Pi_n}(\gn{\varphi}, y)$, we obtain $T + \neg \pi + \neg \varphi \vdash \varphi$, and thus $T + \neg \pi \vdash \varphi$. 
Since $T + \neg \pi \vdash \neg \varphi$, we obtain $T \vdash \pi$. 
Since $\varphi \lor \forall y < \overline{p}\, \neg \Prf_T^{\Pi_n}(\gn{\varphi}, y)$ and $\pi$ are $T$-provable $\Sigma_n$ and $\Pi_n$ sentences respectively, by $\SP$-conservativity, we obtain 
\[
    \PA \vdash \varphi \lor \forall y < \overline{p}\, \neg \Prf_T^{\Pi_n}(\gn{\varphi}, y) \lor \pi.
\]
By combining this with (\ref{eq1}), we have $\PA + \neg \pi \vdash \varphi \lor \varphi$, and hence $\PA + \neg \pi \vdash \varphi$. 
Since $V + \varphi \vdash \pi$, we have $V + \neg \pi$ is inconsistent. 
Therefore, $V \vdash \pi$. 

Let $\sigma$ be any $\Sigma_n$ sentence such that $V + \neg \varphi \vdash \sigma$. 
Since $T + \neg \sigma \vdash \varphi$, we find some $p \in \omega$ such that
\begin{equation}\label{eq2}
    \PA + \neg \sigma \vdash \Prf_T^{\Pi_n}(\gn{\varphi}, \overline{p}).
\end{equation}
By combining this with $T + \varphi \vdash \forall y \leq \overline{p}\, \neg \Prf_T^{\Sigma_n}(\gn{\neg \varphi}, y)$, we obtain $T + \neg \sigma + \varphi \vdash \neg \varphi$, and thus $T + \neg \sigma \vdash \neg \varphi$. 
Since $T + \neg \sigma \vdash \varphi$, we obtain $T \vdash \sigma$. 
Since $\sigma$ and $\neg \varphi \lor \forall y \leq \overline{p}\, \neg \Prf_T^{\Sigma_n}(\gn{\neg \varphi}, y)$ are $T$-provable $\Sigma_n$ and $\Pi_n$ sentences respectively, by $\SP$-conservativity, we obtain 
\[
    \PA \vdash \sigma \lor \neg \varphi \lor \forall y \leq \overline{p}\, \neg \Prf_T^{\Sigma_n}(\gn{\neg \varphi}, y).
\]
By combining this with (\ref{eq2}), we have $\PA + \neg \sigma \vdash \neg \varphi \lor \neg \varphi$, and hence $\PA + \neg \sigma \vdash \neg \varphi$. 
Since $V + \neg \varphi \vdash \sigma$, we have $V + \neg \sigma$ is inconsistent. 
Therefore, $V \vdash \sigma$. 
\end{proof}

We also study simultaneous version of hereditary double $\Pi_n$-conservativity over two theories. 

\begin{thm}\label{HD2}
For any theories $T$ and $U$, the following are equivalent: 
\begin{enumerate}
    \item $\Sigma_n \cap \HDCons(\Pi_n, T) \cap \HDCons(\Pi_n, U)\neq \emptyset$. 

    \item $T$ and $U$ satisfy the following three conditions: 
    \begin{itemize}
        \item $T$ and $U$ are both $\SP$-conservative over $\PA$, 
        \item There exists a $\Sigma_n$ sentence $\varphi$ such that $\varphi \in \HCons(\Pi_n, T)$ and $\neg \varphi \in \HCons(\Sigma_n, U)$,  
        \item There exists a $\Sigma_n$ sentence $\psi$ such that $\psi \in \HCons(\Pi_n, U)$ and $\neg \psi \in \HCons(\Sigma_n, T)$. 
    \end{itemize}
\end{enumerate}
\end{thm}
\begin{proof}
$(1 \Rightarrow 2):$ 
Suppose $\Sigma_n \cap \HDCons(\Pi_n, T) \cap \HDCons(\Pi_n, U)\neq \emptyset$. 
The second and third clauses are obvious, and by Theorem \ref{HD1}, we obtain the first clause.

\medskip

$(2 \Rightarrow 1):$ Let $\varphi$ and $\psi$ be $\Sigma_n$ sentences such that $\varphi \in \HCons(\Pi_n, T)$, $\neg \varphi \in \HCons (\Sigma_n, U)$, $\psi \in \HCons(\Pi_n, U)$, and $\neg \psi \in \HCons (\Sigma_n, T)$. 
Let $\theta_0$ and $\theta_1$ be $\Sigma_n$ sentences satisfying the following equivalences: 
\[
\PA \vdash \theta_0 \leftrightarrow \varphi \wedge \bigl(\PR_{T}^{\Sigma_n}(\gn{\neg (\theta_0 \vee \theta_1)}) \preccurlyeq \PR_{T}^{\Pi_n}(\gn{\theta_0 \vee \theta_1}) \bigr),
\]
\[
\PA \vdash \theta_1 \leftrightarrow \psi \wedge \bigl(\PR_{U}^{\Sigma_n}(\gn{\neg (\theta_0 \vee \theta_1)}) \preccurlyeq \PR_{U}^{\Pi_n}(\gn{\theta_0 \vee \theta_1}) \bigr).
\]
We prove $\theta_0 \vee \theta_1 \in \HDCons(\Pi_n, T) \ \cap \ \HDCons(\Pi_n, U)$. 
We give only a proof of $\theta_0 \vee \theta_1 \in \HDCons(\Pi_n, T)$. 
In particular, we prove only $\theta_0 \vee \theta_1 \in \HCons(\Pi_n, T)$. 
With a slight modification of the proof, $\neg (\theta_0 \vee \theta_1) \in \HCons(\Sigma_n, U)$ is proved similarly.

Let $V$ be any subtheory of $T$ and $\pi \in \Pi_n$ be such that $V + \theta_0 \vee \theta_1 \vdash \pi$.
Then, we obtain $T + \neg \pi \vdash \neg (\theta_0 \vee \theta_1)$, which implies 
\begin{equation}\label{hcons_0}
    \PA + \neg \pi \vdash \Prf_{T}^{\Sigma_n}(\gn{\neg(\theta_0 \vee \theta_1)}, \num{p})
\end{equation}
for some $p \in \omega$.
Since $T + \neg (\theta_0 \vee \theta_1) \vdash \forall z < \num{p} \ \neg \Prf_{T}^{\Pi_n}(\gn{\theta_0 \vee \theta_1}, z)$,
we obtain
\begin{equation}\label{hcons-1}
T \vdash (\theta_0 \vee \theta_1) \vee \forall z < \num{p} \ \neg \Prf_{T}^{\Pi_n}(\gn{\theta_0 \vee \theta_1}, z). 
\end{equation}
From this with $T + \neg \pi \vdash \neg (\theta_0 \lor \theta_1)$, we have $T + \neg \pi \vdash \forall z < \num{p} \ \neg \Prf_{T}^{\Pi_n}(\gn{\theta_0 \vee \theta_1}, z)$. 
By combining this with (\ref{hcons_0}), we obtain $T + \neg \pi \vdash \PR_{U}^{\Sigma_n}(\gn{\neg (\theta_0 \vee \theta_1)}) \preccurlyeq \PR_{U}^{\Pi_n}(\gn{\theta_0 \vee \theta_1})$, which implies that $T + \neg \pi + \varphi  \vdash \theta_0$. 
Since $V + \theta_0 \vdash \pi$, we obtain $T + \varphi \vdash \pi$. 
It follows from $\varphi \in \HCons(\Pi_n,T)$ that $T \vdash \pi$. 

Since the sentence in (\ref{hcons-1}) is $\Sigma_n$ and $\pi$ is $\Pi_n$, by $\SP$-conservativity, we obtain 
\[
\PA \vdash (\theta_0 \vee \theta_1) \vee \forall z < \num{p} \ \neg \Prf_{T}^{\Pi_n}(\gn{\theta_0 \vee \theta_1}, z) \lor \pi.
\]
Since $V \vdash \theta_0 \vee \theta_1 \to \pi$, we have
\[
    V \vdash \forall z < \num{p} \ \neg \Prf_{T}^{\Pi_n}(\gn{\theta_0 \vee \theta_1}, z) \lor \pi.
\]
By combining this with (\ref{hcons_0}), we have $V + \neg \pi \vdash \PR_{U}^{\Sigma_n}(\gn{\neg (\theta_0 \vee \theta_1)}) \preccurlyeq \PR_{U}^{\Pi_n}(\gn{\theta_0 \vee \theta_1})$, and thus $V + \neg \pi + \varphi \vdash \theta_0$. 
Since $V \vdash \theta_0 \to \pi$ again, $V + \varphi \vdash \pi$ holds. 
Since $\varphi \in \HCons(\Pi_n,T)$, we conclude $V \vdash \pi$.
\end{proof}

We also found the following characterization of the condition that appeared in the statement of Theorem \ref{HD2}. 
The following Theorem \ref{HD3} is a generalization of the equivalence $(1 \Leftrightarrow (2 \ \&\ 3))$ in Theorem \ref{HD1}. 

\begin{thm}\label{HD3}
For any theories $T$ and $U$, the following are equivalent: 
\begin{enumerate}
    \item There exists a $\Sigma_n$ sentence $\varphi$ such that $\varphi \in \HCons(\Pi_n, T)$ and $\neg \varphi \in \HCons(\Sigma_n, U)$. 

    \item $\Th_{\Pi_n}(T) \subseteq \HCons(\Sigma_n, U)$ and $\Th_{\Sigma_n}(U) \subseteq \HCons(\Pi_n, T)$. 
\end{enumerate}
\end{thm}
\begin{proof}
$(1 \Rightarrow 2):$
Let $\varphi$ be a $\Sigma_n$ sentence such that $\varphi \in \HCons(\Pi_n, T)$ and $\neg \varphi \in \HCons(\Sigma_n, U)$.
We prove only the inclusion $\Th_{\Pi_n}(T) \subseteq \HCons(\Sigma_n, U)$, and the inclusion $\Th_{\Sigma_n}(U) \subseteq \HCons(\Pi_n, T)$ is proved similarly.
Let $\pi \in \Th_{\Pi_n}(T)$ and $V$ be a subtheory of $U$. 
Suppose that $\sigma$ is a $\Sigma_n$ sentence with $V + \pi \vdash \sigma$. 
Since $\PA + \pi \vee \neg \varphi$ is a subtheory of $T$ and $\PA + \pi  \vee \neg \varphi + \varphi \vdash \pi$, we obtain $\PA + \pi \vee \neg \varphi \vdash \pi$ because $\varphi \in \HCons(\Pi_n, T)$. 
Then, $\PA + \neg \varphi \vdash \pi$ holds. 
Since $V + \pi \vdash \sigma$, it follows that $V + \neg \varphi \vdash \sigma$. 
Since $\neg \varphi \in \HCons(\Sigma_n, U)$, we obtain $V \vdash \sigma$.

\medskip

$(2 \Rightarrow 1):$ Suppose 
$\Th_{\Pi_n}(T) \subseteq \HCons(\Sigma_n, U)$ and $\Th_{\Sigma_n}(U) \subseteq \HCons(\Pi_n, T)$.
Let $\varphi$ be a $\Sigma_n$ sentence satisfying
\[
\PA \vdash \varphi \leftrightarrow \PR_{T}^{\Sigma_n}(\gn{\neg \varphi}) \preccurlyeq \PR_{U}^{\Pi_n}(\gn{\varphi}).
\]
We present only a proof of $\varphi \in \HCons(\Pi_n, T)$, and $\neg \varphi \in \HCons(\Sigma_n, U)$ is proved similarly. 
Let $\pi \in \Pi_n$ and $V$ be a subtheory of $T$. 
Suppose $V+\varphi \vdash \pi$.
Then, we obtain $T + \neg \pi \vdash \neg \varphi$, which implies $\PA + \neg \pi \vdash \Prf_{T}^{\Sigma_n}(\gn{\neg \varphi}, \num{p})$ for some $p \in \omega$.
Thus, $\PA + \neg \pi + \forall z < \num{p} \, \neg \Prf_{U}^{\Pi_n} (\gn{\varphi}, z) \vdash \varphi$ holds. Since $V + \varphi \vdash \pi$, we get 
$V + \forall z < \num{p} \, \neg \Prf_{U}^{\Pi_n} (\gn{\varphi}, z) \vdash \pi$.
We obtain $V + \varphi \vee \ \forall z < \num{p} \, \neg \Prf_{U}^{\Pi_n} (\gn{\varphi}, z) \vdash \pi$.
Since $U + \neg \varphi \vdash \forall z < \num{p} \, \neg \Prf_{U}^{\Pi_n} (\gn{\varphi}, z)$, we have $\varphi \vee \forall z < \num{p} \, \neg \Prf_{U}^{\Pi_n} (\gn{\varphi}, z) \in \Th_{\Sigma_n}(U)$.
By $\Th_{\Sigma_n}(U) \subseteq \HCons(\Pi_n, T)$, it concludes that $V \vdash \pi$.
\end{proof}

The following corollary immediately follows from Theorems \ref{HD1}, \ref{HD2}, and \ref{HD3}. 

\begin{cor}\label{HD4}
For any theories $T$ and $U$, the following are equivalent: 
\begin{enumerate}
    \item $\Sigma_n \cap \HDCons(\Pi_n, T) \cap \HDCons(\Pi_n, U) \neq \emptyset$. 

    \item $\Th_{\Gamma}(S_0) \subseteq \HCons(\Gamma^d, S_1)$ for all $\Gamma \in \{\Sigma_n, \Pi_n\}$ and $S_0, S_1 \in \{T, U\}$.
\end{enumerate}
\end{cor}


\subsection{Simultaneous hereditary conservativity}

First, we study simultaneous exact hereditary $\SP$-conservativity over two theories. 
Proposition \ref{Prop_SP_2} shows that there is no $\Sigma_n$ or $\Pi_n$ sentence which is simultaneously exactly hereditarily $\SP$-conservative. 
Interestingly, as in the cases of $\Sigma_n$-conservativity and $\Pi_n$-conservativity, $\Benp_n$-pairs play a substantial role in the $\SP$ case as well for $\Theta \supseteq \Sigma_n \land \Pi_n$.

\begin{thm}\label{Thm_SP_H}
For any theories $T$ and $U$ and any $n \geq 1$, the following are equivalent: 
\begin{enumerate}
    \item $(T, U)$ is a $\Benp_n$-pair. 
    \item There exists a $\Sigma_n \land \Pi_n$ sentence which is simultaneously exactly hereditarily $\SP$-conservative over $T$ and $U$.
    \item There exists a sentence which is simultaneously exactly hereditarily $\SP$-conservative over $T$ and $U$. 
\end{enumerate}
\end{thm}
\begin{proof}
$(1 \Rightarrow 2)$: 
Suppose $T+ \Th_{\Pi_n}(U)$ and $U+ \Th_{\Pi_n}(T)$ are consistent.
By Bennet's Theorem \ref{Thm_Ben_H}, there exists a $\Sigma_n$ sentence $\varphi$ such that
\[\varphi \in \HCons(\Pi_n, T) \cap \HCons(\Pi_n, U) \setminus (\Th(T) \cup \Th(U)).
\]
By Mostowski's generalization of the first incompleteness theorem (\cite{Mos}), we find a $\Sigma_1$ sentence $\psi$ such that $T + \Th_{\Pi_n}(U) \nvdash \neg \psi$, $U + \Th_{\Pi_n}(T) \nvdash \neg \psi$, $T + \neg \varphi \nvdash \psi$, and $U + \neg \varphi \nvdash \psi$ (cf.~Lindstr\"om \cite[p.~28, Theorem 2.3]{Lin}). 
\paragraph{Claim.}
The theories $T + \varphi \vee \psi + \Th_{\Pi_n}(U+ \varphi \vee \psi)$ and $U + \varphi \vee \psi + \Th_{\Pi_n}(T+ \varphi \vee \psi)$
are consistent. 
\begin{proof}[Proof of Claim]
We only prove that the theory $T + \varphi \vee \psi + \Th_{\Pi_n}(U+ \varphi \vee \psi)$ is consistent.
Suppose, towards a contradiction, that 
$T + \varphi \vee \psi + \Th_{\Pi_n}(U+ \varphi \vee \psi)$ is inconsistent.
Then, there exists a $\Pi_n$ sentence $\pi$ such that $U+ \varphi \vee \psi \vdash \pi$ and $T+ \varphi \vee \psi \vdash \neg \pi$. Thus, we obtain $U + \varphi \vdash \pi$ and $T + \psi \vdash \neg \pi$. Since $\varphi \in \HCons(\Pi_n,U)$, it follows that $U \vdash \pi$. Thus, $T+ \Th_{\Pi_n}(U) \vdash \neg \psi$. 
This is a contradiction. \end{proof}
By Claim and Theorem \ref{Thm_Ben_H}, there exists a $\Pi_n$ sentence $\rho$ such that
\[
\rho \in \HCons(\Sigma_n, T+ \varphi \vee \psi) \cap \HCons(\Sigma_n, U+ \varphi \vee \psi) \setminus (\Th(T) \cup \Th(U)).
\]
We show that the $\Sigma_n \land \Pi_n$ sentence $(\varphi \vee \psi) \wedge \rho$ is simultaneously exactly hereditarily $\SP$-conservative over $T$ and $U$. 

Here, we give a proof of $(\varphi \vee \psi) \wedge \rho  \in \HCons(\SP, T)$, and $(\varphi \vee \psi) \wedge \rho  \in \HCons(\SP, U)$ is proved similarly. 
Let $\sigma \in \Sigma_n$, $\pi \in \Pi_n$, and $V$ be a subtheory of $T$. 
Suppose $V + \varphi \vee \psi + \rho \vdash \sigma \wedge \pi$. Then, $V+ \varphi \vee \psi + \rho \vdash \sigma \preccurlyeq \neg \pi$ holds. Since $\rho \in \HCons(\Sigma_n, T+ \varphi \vee \psi)$, we obtain $V+ \varphi \vee \psi \vdash \sigma \preccurlyeq \neg \pi$, which implies $V+ \varphi \vdash \sigma \preccurlyeq \neg \pi$.
Since $\varphi \in \HCons(\Pi_n, T)$ and $V+ \varphi \vdash \neg (\neg \pi \prec \sigma)$, it follows that $V \vdash \neg (\neg \pi \prec \sigma)$. 
Thus, we obtain $T+ \neg \pi \vdash \sigma \preccurlyeq \neg \pi$, and $T + \neg \pi \vdash \sigma$. 
We conclude $T \vdash \pi \vee \sigma$.

We have $\PA + (\varphi \vee \psi) \wedge \rho \vdash \varphi \vee \psi$ and $\PA + (\varphi \vee \psi) \wedge \rho \vdash \rho$, but $T \nvdash \varphi \vee \psi$, $U \nvdash \varphi \vee \psi$, $T \nvdash \rho$, and $U \nvdash \rho$. 
So, the exactness is verified. 

\medskip

$(2 \Rightarrow 3)$: This implication is trivial.

\medskip

$(3 \Rightarrow 1)$: This follows from Proposition \ref{Prop_SP_1}.(3) and $(6 \Rightarrow 1)$ of Theorem \ref{Thm_char_H_refine}.
\end{proof}

Second, we study simultaneously non-trivially hereditary $\SP$-conservativity. 
For $\Theta \in \{\Sigma_n, \Pi_n\}$, we have the following corollary to Theorem \ref{Thm_Ben_H} and Proposition \ref{Prop_SP_2}.(1). 

\begin{cor}\label{Cor_SP_ntH}
For any theories $T$ and $U$ and any $n \geq 1$, the following are equivalent: 
\begin{enumerate}
    \item $(T, U)$ is a $\Benp_n$-pair. 
    \item There exists a $\Sigma_n$ sentence which is simultaneously non-trivially hereditarily $\SP$-conservative over $T$ and $U$. 
    \item There exists a $\Pi_n$ sentence which is simultaneously non-trivially hereditarily $\SP$-conservative over $T$ and $U$. 
\end{enumerate}
\end{cor}
\begin{proof}
$(1 \Leftrightarrow 2)$: 
By Proposition \ref{Prop_SP_2}.(1), 
\[
    \Sigma_n \cap \HCons(\SP, T) \cap \HCons(\SP, U) \setminus (\Th(T) \cup \Th(U)) \neq \emptyset
\]
is equivalent to 
\[
    \Sigma_n \cap \HCons(\Pi_n, T) \cap \HCons(\Pi_n, U) \setminus (\Th(T) \cup \Th(U)) \neq \emptyset.
\]
By Bennet's Theorem \ref{Thm_Ben_H}, the latter statement holds if and only if $(T, U)$ is a $\Benp_n$-pair. 

$(1 \Leftrightarrow 3)$: This equivalence is proved in the same way by using Proposition \ref{Prop_SP_2}.(2). 
\end{proof}

For the case $\Theta \supseteq \Sigma_n \land \Pi_n$, we introduce the following notion.

\begin{defn}[$\Bend_n$-pairs]\label{Bend_pair}
    Let $T$ and $U$ be any theories and $n \geq 1$.  
We say that $(T, U)$ is a \textit{$\Bend_n$-pair} iff at least one of $T + \Th_{\Pi_n}(U)$ and $U + \Th_{\Pi_n}(T)$ is consistent. 
\end{defn}

Obviously, every $\Benp_n$-pair is a $\Bend_n$-pair. 
But, the converse is not the case in general. 
For example, let $T : = \PA + \neg \PR_{\PA}^{\Sigma_n}(\gn{0=1})$ and $U : = \PA + \PR_{\PA}^{\Sigma_n}(\gn{0=1})$. 
It is known that $\PR_\PA^{\Sigma_n}(\gn{0=1})$ is $\Pi_n$-conservative over $\PA$ (cf.~Blanck \cite[Corollary 4.32]{Bla} and KOSV \cite[Proposition 3.6]{KOSV}), we have that $T + \Th_{\Pi_n}(U)$ is consistent but $U + \Th_{\Pi_n}(T)$ is inconsistent. 
So, $(T, U)$ is a $\Bend_n$-pair, but is not a $\Benp_n$-pair.

The condition that $(T, U)$ is $\Bend_n$-pair seems weak, but interestingly it has the following somewhat stronger consequence.

\begin{thm}\label{Thm_SP_ntH}
For any theories $T$ and $U$ and any $n \geq 1$, the following are equivalent: 
\begin{enumerate}
    \item $(T, U)$ is a $\Bend_n$-pair. 
    \item There exists a $\Sigma_n \land \Pi_n$ sentence which is simultaneously non-trivially hereditarily $\SP$-conservative over $T$ and $U$. 
    \item There exists a sentence which is simultaneously non-trivially hereditarily $\SP$-conservative over $T$ and $U$. 
    \item $\HCons(\SP, T) \setminus \Th(U) \neq \emptyset$. 
\end{enumerate}
\end{thm}
\begin{proof}
$(1 \Rightarrow 2)$: Suppose that $(T, U)$ is a $\Bend_n$-pair. 
Without loss of generality, we may assume $T + \Th_{\Pi_n}(U)$ is consistent.
Then, $(T+ \Th_{\Pi_n}(U)) + \Th_{\Pi_n}(U)$ is consistent. 
By Theorem \ref{Thm_char_H_refine}, there exists a $\Sigma_n$ sentence $\varphi$ such that 
\[
\varphi \in \HCons(\Pi_n, T + \Th_{\Pi_n}(U)) \cap \HCons(\Pi_n, U) \setminus \Th(T+ \Th_{\Pi_n}(U)).
\]
\paragraph{Claim.}
The theory $T+ \varphi + \Th_{\Pi_n}(U + \varphi)$ is consistent.
\begin{proof}[Proof of Claim]
Suppose, towards a contradiction, that $T+ \varphi + \Th_{\Pi_n}(U + \varphi)$ is inconsistent. 
Then, there exists a $\Pi_n$ sentence $\pi$ such that $U + \varphi \vdash \pi$ and $T + \varphi \vdash \neg \pi$.
Since $\varphi \in \HCons(\Pi_n, U)$, we obtain $\pi \in \Th_{\Pi_n}(U)$. 
Hence, $T+ \Th_{\Pi_n}(U) + \varphi$ is inconsistent.
Since $\varphi \in \HCons(\Pi_n, T + \Th_{\Pi_n}(U))$, we have that $T+ \Th_{\Pi_n}(U)$ is inconsistent. 
This is a contradiction.
\end{proof}

By Claim and Theorem \ref{Thm_char_H_refine} again, there exists a $\Pi_n$ sentence $\psi$ such that 
\[
\psi \in \HCons(\Sigma_n, T + \varphi) \cap \HCons(\Sigma_n, U + \varphi) \setminus \Th(U+ \varphi).
\]
We show that the $\Sigma_n \land \Pi_n$ sentence $\varphi \wedge \psi$ is simultaneously non-trivially hereditarily $\SP$-conservative over $T$ and $U$. 
Non-triviality is obvious because $T \nvdash \varphi \wedge \psi$ and $U \nvdash \varphi \wedge \psi$.

We prove $\varphi \wedge \psi \in \HCons(\SP, T) \cap \HCons(\SP, U)$.
We only prove $\varphi \wedge \psi \in \HCons(\SP, T)$, and $\varphi \wedge \psi \in \HCons(\SP, U)$ is proved similarly. 
Let $V$ be any subtheory of $T$, $\sigma \in \Sigma_n$, and $\pi \in \Pi_n$.
Suppose $V + \varphi \wedge \psi \vdash \sigma \wedge \pi$. 
Then, $V + \varphi + \psi \vdash \sigma \preccurlyeq \neg \pi$.
Since $\psi \in \HCons(\Sigma_n, T+\varphi)$,
we obtain $V + \varphi \vdash \sigma \preccurlyeq \neg \pi$, which implies $V + \varphi \vdash \neg (\neg \pi \prec \sigma)$.
Since $\varphi \in \HCons(\Pi_n, T)$, we obtain $V \vdash \neg (\neg \pi \prec \sigma)$. Thus, $V + \neg \pi \vdash \sigma$, that is, $V \vdash \sigma \vee \pi$ holds.

\medskip

$(2 \Rightarrow 3)$ and $(3 \Rightarrow 4)$: Obvious.

\medskip

$(4 \Rightarrow 1)$: 
Let $\varphi \in \HCons(\SP, T) \setminus \Th(U)$.
Suppose, towards a contradiction, that $(T, U)$ is not a $\Bend_n$-pair, that is, both $T + \Th_{\Pi_n}(U)$ and $U + \Th_{\Pi_n}(T)$ are inconsistent. 
We find $\Pi_n$ sentences $\pi_0$ and $\pi_1$ such that $U \vdash \pi_0$, $T \vdash \neg \pi_0$, $T \vdash \pi_1$, and $U \vdash \neg \pi_1$. 
Since $\PA + (\neg \pi_0 \land \pi_1) \lor \neg \varphi$ is a subtheory of $T$ and $\PA + (\neg \pi_0 \land \pi_1) \lor \neg \varphi + \varphi \vdash \neg \pi_0 \land \pi_1$, we have $\PA + (\neg \pi_0 \land \pi_1) \lor \neg \varphi \vdash \neg \pi_0 \lor \pi_1$ because $\varphi \in \HCons(\SP, T)$. 
Then, $\PA + \neg \varphi \vdash \neg \pi_0 \lor \pi_1$ and $\PA + \pi_0 + \neg \pi_1 \vdash \varphi$. 
It follows that $U \vdash \varphi$, a contradiction. 
\end{proof}

\subsection{Simultaneous conservativity}

First, we discuss simultaneous non-trivial $\SP$-conservativity. 
As a consequence of the KOSV theorem (Theorem \ref{Thm_KOSV}) and Proposition \ref{Prop_SP_1}.(2), for any theories $T$ and $U$, there always exists a $\Sigma_n$ sentence which is simultaneously non-trivially $\SP$-conservative over $T$ and $U$. 
For $\Theta = \Pi_n$, as Corollary \ref{Cor_SP_ntH}, we have the following corollary to Theorem \ref{Thm_Ben} and Proposition \ref{Prop_SP_2}.(2). 

\begin{cor}\label{Cor_SP_nt}
For any theories $T$ and $U$ and any $n \geq 1$, the following are equivalent: 
\begin{enumerate}
    \item $(T, U)$ is a $\Ben_n$-pair. 
    \item There exists a $\Pi_n$ sentence which is simultaneously non-trivially $\SP$-conservative over $T$ and $U$. 
\end{enumerate}
\end{cor}

Second, we discuss simultaneous exact $\SP$-conservativity. 
It follows from Proposition \ref{Prop_SP_2} that there is no $\Sigma_n$ or $\Pi_n$ sentence which is simultaneously exactly $\SP$-conservative. 
For $\Theta \supseteq \Sigma_n \land \Pi_n$, we show that the notion of $\Ben_n$-pairs is also a key in this study as in the case of $\Sigma_n$-conservativity.

\begin{thm}\label{Thm_SP}
For any theories $T$ and $U$ and any $n \geq 1$, the following are equivalent: 
\begin{enumerate}
    \item $(T, U)$ is a $\Ben_n$-pair. 
    \item There exists a $\Sigma_n \land \Pi_n$ sentence which is simultaneously exactly $\SP$-conservative over $T$ and $U$. 
    \item There exists a sentence which is simultaneously exactly $\SP$-conservative over $T$ and $U$. 
\end{enumerate}
\end{thm}
\begin{proof}
$(1 \Rightarrow 2)$: Suppose $(T,U)$ is a $\Ben_n$-pair. 
If $(T,U)$ is $\Benp_n$-pair, then we are done by Theorem \ref{Thm_SP_H}. 
If not, we have that $\Th_{\Pi_n}(T) \nsubseteq \Th(U)$ and $\Th_{\Pi_n}(U) \nsubseteq \Th(T)$. 
By the KOSV Theorem \ref{Thm_KOSV}, we obtain a $\Sigma_n$ sentence $\varphi$ such that
\[
    \varphi \in \Cons(\Pi_n, T) \cap \Cons(\Pi_n, U) \setminus (\Th(T) \cup \Th(U)).
\]
Since $\Th_{\Pi_n}(T + \varphi) \nsubseteq \Th(U) = \Th(U + \varphi)$ and $\Th_{\Pi_n}(U + \varphi) \nsubseteq \Th(T) = \Th(T + \varphi)$, we have that $(T+\varphi, U+ \varphi)$ is also a $\Ben_n$-pair.
By Bennet's Theorem \ref{Thm_Ben}, we find a $\Pi_n$ sentence $\psi$ such that
\[
    \psi \in \Cons(\Sigma_n, T+ \varphi) \cap \Cons(\Sigma_n, U+ \varphi) \setminus (\Th(T+ \varphi) \cup \Th(U+ \varphi)).
\]
As in the proof of Theorem \ref{Thm_SP_H}, it is shown that the $\Sigma_n \wedge \Pi_n$ sentence $\varphi \wedge \psi$ is simultaneously exactly $\SP$-conservative over $T$ and $U$.

\medskip

$(2 \Rightarrow 3)$: This implication is obvious.

\medskip

$(3 \Rightarrow 1)$: This implication follows from Proposition \ref{Prop_SP_1}.(3) and Theorem \ref{Thm_char_refine}.  
\end{proof}

\section{$\Delta_n$-conservativity}\label{Sec:Delta}

In this section, we discuss simultaneous hereditary $\Delta_n$-conservativity and simultaneous $\Delta_n$-conservativity. 
For the latter study, we have a perfect picture of the situation, while for the former we do not have a clear picture of the status for some $\Theta$'s.
First of all, we prove the following strengthening of Proposition \ref{Prop_SP_2}.(2). 

\begin{prop}\label{Prop_D_1}
Let $n \geq 1$. 
If $\varphi$ is a $\Pi_n$ sentence which is $\Delta_n$-conservative over $T$, then $\varphi$ is $\Sigma_n$-conservative over $T$. 
So, there is no $\Pi_n$ sentence which is exactly $\Delta_n$-conservative over $T$.
\end{prop}
\begin{proof}
Let $\varphi$ be a $\Pi_n$ sentence which is $\Delta_n$-conservative over $T$.
Let $\sigma \in \Sigma_n$ be such that $T+ \varphi \vdash \sigma$. 
Then, we obtain $T + \varphi \vdash \sigma \preccurlyeq \neg \varphi$.
Since $T \vdash \neg \varphi \vee \sigma$, we obtain $T \vdash \sigma \preccurlyeq \neg \varphi \leftrightarrow \neg (\neg \varphi \prec \sigma)$, and so $\sigma \preccurlyeq \neg \varphi$ is $\Delta_n(T)$. 
It follows that $T \vdash \sigma \preccurlyeq \neg \varphi$, which implies $T \vdash \sigma$.
\end{proof}

\subsection{Simultaneous hereditary conservativity}

We study simultaneous exact hereditary $\Delta_n$-conservativity. 
Proposition \ref{Prop_D_1} shows that there is no $\Pi_n$ sentence which is simultaneously exactly hereditarily $\Delta_n$-conservative. 
For $\Theta \supseteq \Sigma_n$, we found a necessary condition and a sufficient condition for the existence of a $\Theta$ sentence which is simultaneously exactly hereditarily $\Delta_n$-conservative over $T$ and $U$. 
For this investigation, we introduce the following notion, which will be considered in Section \ref{Sec:and} again. 

\begin{defn}[$\Bena_n$-pairs]\label{Bena_pair}
Let $T$ and $U$ be any theories and $n \geq 1$.  
We say that $(T, U)$ is a \textit{$\Bena_n$-pair} iff both $T + \Th_{\Sigma_n \land \Pi_n}(U)$ and $U + \Th_{\Sigma_n \land \Pi_n}(T)$ are consistent. 
\end{defn}

Obviously, every $\Bena_n$-pair is a $\Benp_n$-pair, but the converse is generally not the case. 

\begin{prop}
For every $n \geq 1$, there exist theories $T$ and $U$ such that $(T, U)$ is a $\Benp_n$-pair but not a $\Bena_n$-pair. 
\end{prop}
\begin{proof}
By Guaspari's Theorem \ref{Thm_Gua}, we find a sentence $\alpha \in \Pi_n \cap \Cons(\Sigma_n, \PA) \setminus \Th(\PA)$. 
Since $\PA + \neg \alpha$ is consistent, we find a sentence $\beta \in \Pi_n \cap \Cons(\Sigma_n, \PA + \neg \alpha) \setminus \Th(\PA + \neg \alpha)$ by Theorem \ref{Thm_Gua} again. 
Let $T : = \PA + \alpha \lor \neg \beta$ and $U : = \PA + \neg \alpha + \beta$. 
It is easy to see that $T$ and $U$ are consistent. 
Since $T + \Th_{\Sigma_n \land \Pi_n}(U)$ is trivially inconsistent, $(T, U)$ is not a $\Bena_n$-pair. 
So, it suffices to show that $(T, U)$ is a $\Benp_n$-pair. 

Suppose, towards a contradiction, that $T + \Th_{\Pi_n}(U)$ is inconsistent. 
Then, there is a $\Pi_n$ sentence $\pi$ such that $U \vdash \pi$ and $T \vdash \neg \pi$. 
By the definition of $T$, we have $\PA + \alpha \vdash \neg \pi$. 
Since $\neg \pi$ is $\Sigma_n$, we obtain $\PA \vdash \neg \pi$. 
This contradicts the consistency of $U$. 
Therefore, $T + \Th_{\Pi_n}(U)$ is consistent. 

Suppose, towards a contradiction, that $U + \Th_{\Pi_n}(U)$ is inconsistent. 
Then, there is a $\Pi_n$ sentence $\pi$ such that $T \vdash \pi$ and $U \vdash \neg \pi$. 
Since $\PA + \neg \alpha + \beta \vdash \neg \pi$ and $\neg \pi$ is $\Sigma_n$, we have $\PA + \neg \alpha \vdash \neg \pi$. 
By the definition of $T$, we have $\PA + \neg \beta \vdash \pi$. 
By combining them, we obtain $\PA + \neg \alpha \vdash \beta$, a contradiction. 
Hence, $U + \Th_{\Pi_n}(T)$ is consistent. 

We have shown that $(T, U)$ is a $\Benp_n$-pair. 
\end{proof}


\begin{thm}\label{Thm_D_H}
Let $T$ and $U$ be any theories and $n \geq 1$.  
\begin{enumerate}
    \item If $(T, U)$ is a $\Bena_n$-pair, then there exists a $\Sigma_n$ sentence which is simultaneously exactly hereditarily $\Delta_n$-conservative over $T$ and $U$. 
    \item If there exists a sentence which is simultaneously exactly hereditarily $\Delta_n$-conservative over $T$ and $U$, then $(T, U)$ is a $\Benp_n$-pair. 
\end{enumerate}
\end{thm}
\begin{proof}
1. Suppose $(T,U)$ is a $\Bena_n$-pair. 
Since $(T + \Th_{\Pi_n}(U)) + \Th_{\Sigma_n}(U)$ is consistent, we have that $U + \Th_{\Pi_n}(T + \Th_{\Pi_n}(U))$ is consistent. 
Equivalently, $U + \Th_{\Pi_n}(T) + \Th_{\Pi_n}(T + \Th_{\Pi_n}(U))$ is consistent. 
We also have that $T + \Th_{\Pi_n}(U) + \Th_{\Pi_n}(U + \Th_{\Pi_n}(T))$ is consistent. 
By Theorem \ref{Thm_Ben_H}, we obtain a $\Pi_n$ sentence $\varphi$ such that
\[
\varphi \in \HCons(\Sigma_n, T) \cap \HCons(\Sigma_n, U) \setminus \bigl(\Th(T+ \Th_{\Pi_n}(U)) \cup \Th(U+ \Th_{\Pi_n}(T)) \bigr).
\]
By the Fixed Point Theorem, we find $\Sigma_n$ sentences $\sigma_0$ and $\sigma_1$ such that
\[
\PA \vdash \sigma_0 \leftrightarrow \exists x \psi(x) \preccurlyeq \neg \varphi \ \text{and} \ 
\PA \vdash \sigma_1 \leftrightarrow \exists x \rho(x) \preccurlyeq \neg \varphi,
\]
where $\psi(x)$ and $\rho(x)$ are $\Pi_{n-1}$ formulas such that 
\begin{align*}
    \PA & \vdash \exists x \psi(x) \leftrightarrow \PR_{T}^{\Sigma_n}(\gn{\neg \sigma_0}) \preccurlyeq (\PR_{T+ \neg \varphi}(\gn{\sigma_0 \vee \sigma_1}) \vee \PR_{U+ \neg \varphi}(\gn{\sigma_0 \vee \sigma_1})), \\
    \PA & \vdash \exists x \rho(x) \leftrightarrow \PR_{U}^{\Sigma_n}(\gn{\neg \sigma_1}) \preccurlyeq (\PR_{T+ \neg \varphi}(\gn{\sigma_0 \vee \sigma_1}) \vee \PR_{U+ \neg \varphi}(\gn{\sigma_0 \vee \sigma_1})).
\end{align*}

\paragraph{Claim 1.}
$T+ \neg \varphi \nvdash \sigma_0 \vee \sigma_1$ and $U+ \neg \varphi \nvdash \sigma_0 \vee \sigma_1$.
\begin{proof}[Proof of Claim 1]
We only prove $T + \neg \varphi \nvdash \sigma_0 \vee \sigma_1$, and $U+ \neg \varphi \nvdash \sigma_0 \vee \sigma_1$ is proved similarly.
Suppose, towards a contradiction, that $T + \neg \varphi \vdash \sigma_0 \vee \sigma_1$. Then, 
we obtain $U + \sigma_1 \vdash (\PR_{T + \neg \varphi}(\gn{\sigma_0 \vee \sigma_1}) \lor \PR_{U + \neg \varphi}(\gn{\sigma_0 \vee \sigma_1})) \prec \PR_{U}^{\Sigma_n}(\gn{\neg \sigma_1})$, which implies $U+ \sigma_1 \vdash \neg \exists x \rho(x)$. Hence, we obtain $U+ \sigma_1 + \neg \varphi \vdash \neg \varphi \prec \exists x \rho(x)$. 
So, $U + \sigma_1 + \neg \varphi \vdash \neg \sigma_1$, that is, $U+ \neg \varphi \vdash \neg \sigma_1$. 
By the similar argument, we obtain $T+ \neg \varphi \vdash \neg \sigma_0$, which implies $T+ \neg \varphi \vdash \sigma_1$ because $T + \neg \varphi \vdash \sigma_0 \vee \sigma_1$. 
Since $U \vdash \varphi \lor \neg \sigma_1$ and $T+ \neg \varphi + \varphi \lor \neg \sigma_1 \vdash \sigma_1 \wedge \neg \sigma_1$,
we obtain that $T+\Th_{\Pi_n}(U) + \neg \varphi$ is inconsistent. 
That is, $T+ \Th_{\Pi_n}(U) \vdash \varphi$ holds. 
This contradicts $\varphi \notin \Th(T+ \Th_{\Pi_n}(U))$.
\end{proof}

\paragraph{Claim 2.}
$\sigma_0 \vee \sigma_1 \in \Sigma_n \cap \HCons(\Delta_n, T) \cap \HCons(\Delta_n. U)$.

\begin{proof}[Proof of Claim 2]
We only prove $\sigma_0 \vee \sigma_1 \in \HCons(\Delta_n, T)$. 
Let $V$ be a subtheory of $T$ and let $\delta \in \Delta_n(V)$ be such that $V+ \sigma_0 \vee \sigma_1 \vdash \delta$. 
Then, $V+ \sigma_0 \vdash \delta$ and $T + \neg \delta \vdash \neg \sigma_0$.
Since $\neg \delta \in \Sigma_n(T)$, there exists a $p \in \omega$
such that $\PA + \neg \delta \vdash \Prf_{T}^{\Sigma_n}(\gn{\neg \sigma_0}, \num{p})$. 
By combining this with Claim 1, we obtain $\PA + \neg \delta \vdash \exists x \psi(x)$, which implies $V+\neg \delta + \varphi \vdash \sigma_0$. 
It follows from $V \vdash \sigma_0 \to \delta$ that $V+ \varphi \vdash \delta$. 
Since $\delta \in \Sigma_n(V)$ and $\varphi \in \HCons(\Sigma_n, T)$, we conclude $V \vdash \delta$.
\end{proof}

\paragraph{Claim 3.}
There exists a $\Pi_n$ sentence $\pi$ such that $\PA + \sigma_0 \vee \sigma_1 \vdash \pi$, $T \nvdash \pi$, and $U \nvdash \pi$.
\begin{proof}[Proof of Claim 3]
We take $\pi \equiv \neg (\neg \varphi \prec \exists x \psi(x)) \vee \neg (\neg \varphi \prec \exists x \rho (x))$ as a desired $\Pi_n$ sentence.
We easily obtain $\PA + \sigma_0 \vee \sigma_1 \vdash \pi$.
We prove $T \nvdash \pi$.
Suppose, towards a contradiction, that $T \vdash \pi$.
Then, we obtain $T+ \neg \varphi \vdash (\exists x \psi(x) \preccurlyeq \neg \varphi)  \vee (\exists x \rho (x) \preccurlyeq \neg \varphi)$, which implies $T + \neg \varphi \vdash \sigma_0 \vee \sigma_1$. This contradicts
Claim 1. It is similarly proved that $U \nvdash \pi$.
\end{proof}
By Claims 2 and 3, the $\Sigma_n$ sentence $\sigma_0 \vee \sigma_1$ is simultaneously exactly hereditarily $\Delta_n$-conservative over $T$ and $U$.

\medskip

2. This is a direct consequence of Theorem \ref{Thm_char_H_refine}. 
\end{proof}

\begin{prob}\label{Prob_D_H}
For $\Theta \supseteq \Sigma_n$, does there exist a reasonable condition of $(T, U)$ that is equivalent to the existence of a $\Theta$ sentence which is simultaneously exactly hereditarily $\Delta_n$-conservative over $T$ and $U$?
\end{prob}

Next, we investigate simultaneous non-trivial hereditary $\Delta_n$-conservativity. 
For $\Theta = \Pi_n$, the following corollary follows from Proposition \ref{Prop_D_1} and Theorem \ref{Thm_Ben_H}. 

\begin{cor}\label{Cor_D_ntH}
For any theories $T$ and $U$ and any $n \geq 1$, the following are equivalent: 
\begin{enumerate}
    \item $(T, U)$ is a $\Benp_n$-pair. 
    \item There exists a $\Pi_n$ sentence which is simultaneously non-trivially hereditarily $\Delta_n$-conservative over $T$ and $U$. 
\end{enumerate}
\end{cor}


Here, we discuss the existence of a $\Theta$ sentence which is simultaneously non-trivially hereditarily $\Delta_n$-conservative over two theories for $\Theta \supseteq \Sigma_n$. 
For $\Theta \supseteq \Sigma_n \land \Pi_n$, the following Theorem \ref{Thm_D_ntH} tells us that the existence of such a sentence is equivalent that $(T, U)$ is a $\Bend_n$-pair. 

The case $\Theta = \Sigma_n$ is somewhat problematic for us. 
If $(T, U)$ is a $\Benp_n$-pair, then the existence of a $\Sigma_n$ sentence $\varphi$ that is simultaneously non-trivially hereditarily $\Pi_n$-conservative over $T$ and $U$ follows from Bennet's Theorem \ref{Thm_Ben_H}. 
On the other hand, the following theorem states that the existence of a sentence which is simultaneously non-trivially hereditarily $\Delta_n$-conservative over $T$ and $U$ implies that $(T, U)$ is a $\Bend_n$-pair. 
So, the condition of the existence of a $\Sigma_n$ sentence which is simultaneously non-trivially hereditarily $\Delta_n$-conservative over $T$ and $U$ intermediates between those that $(T, U)$ is a $\Benp_n$-pair and that $(T, U)$ is a $\Bend_n$-pair. 
However, the equivalence $(1 \Leftrightarrow 5)$ of the following Theorem \ref{Thm_D_ntH} states that $\Bend_n$-pair only guarantees the existence of a $\Sigma_n$ sentence which is simultaneously non-trivially hereditarily $\Delta_n$-conservative in a weak sense.

\begin{thm}\label{Thm_D_ntH}
For any theories $T$ and $U$ and any $n \geq 1$, the following are equivalent: 
\begin{enumerate}
    \item $(T, U)$ is a $\Bend_n$-pair.
    \item There exists a $\Sigma_n \land \Pi_n$ sentence which is simultaneously non-trivially hereditarily $\Delta_n$-conservative over $T$ and $U$. 
    \item There exists a sentence which is simultaneously non-trivially hereditarily $\Delta_n$-conservative over $T$ and $U$.
    \item $\HCons(\Delta_n, T) \setminus \Th(U) \neq \emptyset$. 
    \item There exists a $\Sigma_n$ sentence $\varphi$ such that for any subtheory $V$ of $T$ or $U$ with $V \vdash \Th(T) \cap \Th(U)$, $\varphi$ is $\Delta_n(\Th(T) \cap \Th(U))$-conservative over $V$, $T \nvdash \varphi$, and $U \nvdash \varphi$. 
    \item There exists a sentence $\varphi$ such that for any subtheory $V$ of $T$ with $V \vdash \Th(T) \cap \Th(U)$, $\varphi$ is $\Delta_n(\Th(T) \cap \Th(U))$-conservative over $V$ and $U \nvdash \varphi$. 
\end{enumerate}
\end{thm}
\begin{proof}
$(1 \Rightarrow 2)$: This implication follows from Proposition \ref{Prop_SP_1}.(3) and Theorem \ref{Thm_SP_ntH}. 

\medskip

$(2 \Rightarrow 3)$, $(3 \Rightarrow 4)$, and $(4 \Rightarrow 6)$: Obvious. 

\medskip

$(1 \Rightarrow 5)$: 
Let $(T,U)$ is a $\Bend_n$-pair.
Without loss of generality, we assume that $U + \Th_{\Pi_n}(T)$ is consistent.
Then, $T+\Th_{\Sigma_n}(U)$ is consistent.

Let $S$ be either $T$ or $U$. 
We define the $\Sigma_n$ formula $\Prf_S^{\Delta_n}(x, y)$ as follows: 
\begin{align*}
\Prf_{S}^{\Delta_n} (x,y) : \equiv & \exists u, v, w \leq y \Bigl(\Sigma_n(u) \wedge \Pi_n(v) \\
& \quad \wedge \Prf_{\Th(T) \cap \Th(U)}(u \dot{\leftrightarrow} v, w) \wedge \True_{\Sigma_n}(u) \wedge \Prf_{T}(u \dot{\to} x, y)\Bigr). 
\end{align*}
By Craig's trick, we may assume that the proof predicate $\Prf_{\Th(T) \cap \Th(U)}(x, y)$ of $\Th(T) \cap \Th(U)$ is $\Delta_1(\PA)$. 
Let $\PR_{S}^{\Delta_n}(x) : \equiv \exists y \Prf_{S}^{\Delta_n} (x,y)$.
Then, the following claim is proved as Proposition \ref{ref}, so we omit its proof. 

\paragraph{Claim 1.}
Let $\psi$ be any sentence.
\begin{enumerate}
\item 
For any $p \in \omega$, we have $S \vdash \Prf_{S}^{\Delta_n}(\gn{\psi}, \num{p}) \to \psi$.

\item
For any $\delta \in \Delta_n (\Th(T) \cap \Th(U))$, if $S+\delta \vdash \psi$, then 
$\Th(T) \cap \Th(U) + \delta \vdash \Prf_{S}^{\Delta_n}(\gn{\psi}, \num{p})$ for some $p \in \omega$.
\end{enumerate}

By the Fixed Point Theorem, we find $\Sigma_n$ sentences $\sigma_0$ and $\sigma_1$ satisfying the following equivalences:
\begin{align*}
    \PA & \vdash \sigma_0 \leftrightarrow \PR_{T}^{\Delta_n}(\gn{\neg \sigma_0}) \preccurlyeq (\PR_{T}(\gn{\sigma_0 \vee \sigma_1}) \vee \PR_{U}(\gn{\sigma_0 \vee \sigma_1})), \\
    \PA & \vdash \sigma_1 \leftrightarrow \PR_{U}^{\Delta_n}(\gn{\neg \sigma_1}) \preccurlyeq (\PR_{T}(\gn{\sigma_0 \vee \sigma_1}) \vee \PR_{U}(\gn{\sigma_0 \vee \sigma_1})).
\end{align*}

\paragraph{Claim 2.}
$U \nvdash \sigma_0 \vee \sigma_1$.

\begin{proof}[Proof of Claim 2]
Suppose, towards a contradiction, that $U \vdash \sigma_0 \vee \sigma_1$. 
Then, there exists a $p \in \omega$ such that $\PA \vdash \Prf_{U}(\gn{\sigma_0 \vee \sigma_1}, \overline{p})$. 
By Claim 1, we have $U + \sigma_1 \vdash \forall z \leq \overline{p}\, \neg \Prf_{U}^{\Delta_n}(\gn{\neg \sigma_1}, z)$. 
Then, we obtain $U + \sigma_1 \vdash \neg \sigma_1$, and hence $U \vdash \neg \sigma_1$.
So, $U \vdash \sigma_0$.
In the same way, we obtain $T \vdash \neg \sigma_0$. 
Then, $U+ \Th_{\Pi_n}(T)$ is inconsistent. 
This is a contradiction.
\end{proof}

\paragraph{Claim 3.}
$T \nvdash \sigma_0 \vee \sigma_1$.
\begin{proof}[Proof of Claim 3]
Suppose, towards a contradiction, that $T \vdash \sigma_0 \vee \sigma_1$.
Let $p \in \omega$ be the least $T$-proof of $\sigma_0 \lor \sigma_1$. 
Then, $\PA \vdash \Prf_T(\gn{\sigma_0 \vee \sigma_1}, \num{p})$. 
As in the proof of Claim 2, we have $T \vdash \sigma_1$ and $U \vdash \neg \sigma_1$.
Also, by the definition of $\sigma_1$, 
\begin{align*}
    \PA \vdash \sigma_1 & \leftrightarrow \exists y \leq \overline{p}\, \Prf_U^{\Delta_n}(\gn{\neg \sigma_1}, y), \\
    & \leftrightarrow \exists y \leq \overline{p}\, \exists u, v, w \leq y \Bigl(\Sigma_n(u) \wedge \Pi_n(v) \\
& \qquad \wedge \Prf_{\Th(T) \cap \Th(U)}(u \dot{\leftrightarrow} v, w) \wedge \True_{\Sigma_n}(u) \wedge \Prf_{T}(u \dot{\to} x, y)\Bigr). 
\end{align*}

For each $q \leq p$, let $(\theta_0 \leftrightarrow \xi_0), \ldots ,(\theta_m \leftrightarrow \xi_m)$ be an enumeration of all sentences of the form $\theta \leftrightarrow \xi$ for some $\theta \in \Sigma_n$ and $\xi \in \Pi_n$, having a $\Th(T) \cap \Th(U)$-proof that is less than or equal to $q$. 
Then, we obtain 
\begin{align*}
\PA \vdash u, v, w \leq \num{q} & \wedge \Sigma_n(u) \wedge \Pi_n(v) \wedge \Prf_{\Th(T) \cap \Th(U)}(u \dot{\leftrightarrow} v, w) \notag \\
& \to \bigvee_{i \leq m}(u=\gn{\theta_i} \wedge v=\gn{\xi_i}).
\end{align*}
Since $\Th(T) \cap \Th(U) \vdash \bigwedge_{i \leq m} (\theta_i \leftrightarrow \xi_i)$, we have
\begin{align*}
\Th(T) \cap \Th(U) \vdash u, v, w \leq \num{q} & \wedge \Sigma_n(u) \wedge \Pi_n(v) \wedge \Prf_{\Th(T) \cap \Th(U)}(u \dot{\leftrightarrow} v, w) \notag \\
& \to \Bigl(\True_{\Sigma_n}(u) \leftrightarrow \True_{\Pi_n}(v) \Bigr).
\end{align*}
So, $\sigma_1$ is equivalent to 
\begin{align*}
    \exists y \leq \overline{p}\, \exists u, v, w \leq y \Bigl(& \Sigma_n(u) \wedge \Pi_n(v) \\
    & \wedge \Prf_{\Th(T) \cap \Th(U)}(u \dot{\leftrightarrow} v, w) \wedge \True_{\Pi_n}(v) \wedge \Prf_{T}(u \dot{\to} x, y)\Bigr) 
\end{align*}
over $\Th(T) \cap \Th(U)$. 
It follows that $\sigma_1$ is $\Pi_n$ over $\Th(T) \cap \Th(U)$, and thus $\neg \sigma_1$ is $\Sigma_n$ over $\Th(T) \cap \Th(U)$.
Hence, $T +\Th_{\Sigma_n}(U)$ is inconsistent.
This is a contradiction.
\end{proof}

Let $V$ be any subtheory of $T$ or $U$ with $V \vdash \Th(T) \cap \Th(U)$. 
It suffices to show that $\sigma_0 \vee \sigma_1$ is $\Delta_n(\Th(T) \cap \Th(U))$-conservative over $V$.
We only prove the case that $V$ is a subtheory of $T$.
Let $\delta \in \Delta_n(\Th(T) \cap \Th(U))$ be such that $V+ \sigma_0 \vee \sigma_1 \vdash \delta$. 
Then, $V+ \sigma_0 \vdash \delta$, which implies $T + \neg \delta \vdash \neg \sigma_0$. 
By Claim 1, there exists a $q \in \omega$ such that $\Th(T) \cap \Th(U) + \neg \delta \vdash \Prf_{T}^{\Delta_n}(\gn{\neg \sigma_0}, \num{q})$.
By the definition of $\sigma_0$ and Claims 2 and 3, $\Th(T) \cap \Th(U) + \neg \delta \vdash \sigma_0$.
Since $V + \sigma_0 \vdash \delta$, we obtain $V \vdash \delta$.

\medskip

$(5 \Rightarrow 6)$: Obvious.

\medskip

$(6 \Rightarrow 1)$: 
Let $\varphi$ be a sentence satisfying the conditions stated in the sixth clause. 
Suppose, towards a contradiction, that the theories $T+ \Th_{\Pi_n}(U)$ and $U+ \Th_{\Pi_n}(T)$ are inconsistent.
Then, there exist $\Pi_n$ sentences $\pi_0$ and $\pi_1$ such that $U \vdash \pi_0$, $T \vdash \neg \pi_0$, $T \vdash \pi_1$, and $U \vdash \neg \pi_1$. 
Since $\Th(T) \cap \Th(U) \vdash \neg \pi_0 \vee \neg \pi_1$, the sentence $\neg \pi_0 \preccurlyeq \neg \pi_1$ is $\Delta_n(\Th(T) \cap \Th(U))$.
We obtain $T \vdash \neg \pi_0 \preccurlyeq \neg \pi_1$ and $U \vdash \neg (\neg \pi_0 \preccurlyeq \neg \pi_1)$.
Since $\Th(T) \cap \Th(U) + (\neg \pi_0 \preccurlyeq \neg \pi_1) \vee \neg \varphi$ is a subtheory of $T$ and $\Th(T) \cap \Th(U) + (\neg \pi_0 \preccurlyeq \neg \pi_1) \vee \neg \varphi + \varphi \vdash \neg \pi_0 \preccurlyeq \neg \pi_1$,
it follows that $\Th(T) \cap \Th(U) + (\neg \pi_0 \preccurlyeq \neg \pi_1) \vee \neg \varphi \vdash \neg \pi_0 \preccurlyeq \neg \pi_1$, and hence $\Th(T) \cap \Th(U) + \neg \varphi \vdash \neg \pi_0 \preccurlyeq \neg \pi_1$.
Since $U \vdash \neg(\neg \pi_0 \preccurlyeq \neg \pi_1)$, we obtain $U \vdash \varphi$.
This is a contradiction.
\end{proof}

Notice that the implication $(6 \Rightarrow 1)$ of Theorem \ref{Thm_D_ntH} is a refinement of the implication $(4 \Rightarrow 1)$ of Theorem \ref{Thm_SP_ntH}. 

\begin{prob}\label{Prob_D_ntH}
Does there exist a reasonable condition of $(T, U)$ that is equivalent to the existence of a $\Sigma_n$ sentence which is simultaneously non-trivially hereditarily $\Delta_n$-conservative over $T$ and $U$?
\end{prob}

\begin{rem}
Let $T : = \PA + \neg \PR_{\PA}^{\Sigma_n}(\gn{0=1})$ and $U : = \PA + \PR_{\PA}^{\Sigma_n}(\gn{0=1})$. 
We noted in the paragraph directly below Definition \ref{Bend_pair} that $(T, U)$ is a $\Bend_n$-pair but not a $\Benp_n$-pair. 
It follows from Theorem \ref{Thm_D_ntH} that there exists a $\Sigma_n$ sentence $\varphi$ such that for any subtheory $V$ of $T$ or $U$, $\varphi$ is $\Delta_n(\PA)$-conservative over $V$, $T \nvdash \varphi$, and $U \nvdash \varphi$. 
We do not know whether there exists a $\Sigma_n$ sentence $\psi$ which is simultaneously non-trivially hereditarily $\Delta_n$-conservative over $T$ and $U$. 
If such a sentence $\psi$ exists, then $\PA + \psi$ is $\Pi_n$-conservative over $T$ by Theorem \ref{Thm_char_H_refine} because $U + \Th_{\Pi_n}(T)$ is inconsistent. 
\end{rem}

\subsection{Simultaneous conservativity}

Here, we study simultaneous non-trivial $\Delta_n$-conservativity. 
It follows from the KOSV Theorem that for any theories $T$ and $U$, there always exists a $\Sigma_n$ sentence which is simultaneously non-trivially $\Delta_n$-conservative over $T$ and $U$. 
As in the case of $\SP$-conservativity, for $\Theta = \Pi_n$, the following corollary follows from Proposition \ref{Prop_D_1} and Theorem \ref{Thm_Ben}. 

\begin{cor}\label{Cor_D_nt}
For any theories $T$ and $U$ and any $n \geq 1$, the following are equivalent: 
\begin{enumerate}
    \item $(T, U)$ is a $\Ben_n$-pair. 
    \item There exists a $\Pi_n$ sentence which is simultaneously non-trivially $\Delta_n$-conservative over $T$ and $U$. 
\end{enumerate}
\end{cor}

Then, we discuss simultaneous exact $\Delta_n$-conservativity. 
Proposition \ref{Prop_D_1} implies that there is no $\Pi_n$ sentence which is simultaneously exactly $\Delta_n$-conservative. 
For $\Theta \supseteq \Sigma_n$, we prove the following theorem.

\begin{thm}\label{Thm_D}
For any theories $T$ and $U$ and any $n \geq 1$, the following are equivalent: 
\begin{enumerate}
    \item $(T, U)$ is a $\Ben_n$-pair. 
    \item There exists a $\Sigma_n$ sentence which is simultaneously exactly $\Delta_n$-conservative over $T$ and $U$. 
    \item There exists a sentence which is simultaneously exactly $\Delta_n$-conservative over $T$ and $U$. 
\end{enumerate}
\end{thm}
\begin{proof}
$(1 \Rightarrow 2)$: 
Suppose that $(T, U)$ is a $\Ben_n$-pair. 
By Bennet's Theorem \ref{Thm_Ben}, we find a $\Pi_n$ sentence $\varphi$ such that
\[
    \varphi \in \Cons(\Sigma_n, T) \cap \Cons(\Sigma_n, U) \setminus (\Th(T) \cup \Th(U)). 
\]
By the KOSV Theorem \ref{Thm_KOSV}, we also find a $\Sigma_n$ sentence $\psi$ such that
\[
    \psi \in \Cons(\Pi_n, T + \varphi) \cap \Cons(\Pi_n, U + \varphi) \setminus (\Th(T + \neg \varphi) \cup \Th(U + \neg \varphi)). 
\]
We prove that the $\Sigma_n$ sentence $\psi \preccurlyeq \neg \varphi$ is simultaneously exactly $\Delta_n$-conservative over $T$ and $U$. 

We only prove $(\psi \preccurlyeq \neg \varphi) \in \Cons(\Delta_n, T)$, and $(\psi \preccurlyeq \neg \varphi) \in \Cons(\Delta_n, U)$ is proved in the same way. 
Let $\delta$ be any $\Delta_n(T)$ sentence such that $T + \psi \preccurlyeq \neg \varphi \vdash \delta$. 
Then, $T + \varphi + \psi \vdash \delta$. 
Since $\delta$ is $\Pi_n(T)$, we have $T + \varphi \vdash \delta$ because $\psi \in \Cons(\Pi_n, T + \varphi)$. 
Also, since $\delta$ is $\Sigma_n(T)$, we get $T \vdash \delta$ because $\varphi \in \Cons(\Sigma_n, T)$. 

We prove the exactness, namely, $(\psi \preccurlyeq \neg \varphi) \notin \Cons(\SP, T) \cup \Cons(\SP, U)$.
We only give a proof of $(\psi \preccurlyeq \neg \varphi) \notin \Cons(\SP, T)$. 
Since $T + \psi \preccurlyeq \neg \varphi \vdash (\psi \preccurlyeq \neg \varphi) \land \neg (\neg \varphi \prec \psi)$, it suffices to prove that $T \nvdash \neg (\neg \varphi \prec \psi)$, which implies $T \nvdash (\psi \preccurlyeq \neg \varphi) \lor \neg (\neg \varphi \prec \psi)$. 
Suppose, towards a contradiction, that $T \vdash \neg (\neg \varphi \prec \psi)$. 
Then, we have $T + \neg \varphi \vdash \psi$, a contradiction. 

\medskip

$(2 \Rightarrow 3)$: Trivial. 

\medskip

$(3 \Rightarrow 1)$: This implication follows from Theorem \ref{Thm_char_refine}. 
\end{proof}

\section{$\mathcal{B}(\Sigma_n)$-conservativity}\label{Sec:Boole}

The first result concerning $\mathcal{B}(\Sigma_n)$-conservative sentences in the literature is the following theorem due to H\'ajek.

\begin{thm}[H\'ajek {\cite[Theorem 2]{Haj79}}]\label{Thm_Haj}
Let $T$ be any theory and $n \geq 1$. 
There exists a $\Delta_{n+1}(\PA)$ sentence $\varphi$ such that $\varphi \in \Cons(\mathcal{B}(\Sigma_n), T)$ and $\neg \varphi \in \Cons(\Pi_n, T)$. 
\end{thm}

Compared to the situations in the previous sections, the situation with respect to $\mathcal{B}(\Sigma_n)$-conservativity is simpler.
We will show that simultaneous non-trivial hereditary $\mathcal{B}(\Sigma_n)$-conservativity is equivalent to simultaneous exact hereditary $\mathcal{B}(\Sigma_n)$-conservativity. 
We will also prove that for any $T$ and $U$, there always exists a $\Delta_{n+1}(\PA)$ sentence which is simultaneously non-trivially $\mathcal{B}(\Sigma_n)$-conservative over $T$ and $U$. 

First, we discuss simultaneous hereditary $\mathcal{B}(\Sigma_n)$-conservativity. 
We introduce the following notion. 

\begin{defn}[$\Benb_n$-pairs]\label{Benb_pair}
Let $T$ and $U$ be any theories and $n \geq 1$.  
We say that $(T, U)$ is a \textit{$\Benb_n$-pair} iff $T + \Th_{\mathcal{B}(\Sigma_n)}(U)$ is consistent. 
\end{defn}

Obviously, every $\Benb_n$-pair is a $\Bena_n$-pair. 
Also, it is shown that every $\Bend_{n+1}$-pair is a $\Benb_n$-pair. 

We prove the following theorem based on H\'ajek's idea of using the formula $\PR_T^{\Sigma_n \land \Pi_n}(x)$. 

\begin{thm}\label{Thm_Boole_H}
For any theories $T$ and $U$ and any $n \geq 1$, the following are equivalent: 
\begin{enumerate}
    \item $(T, U)$ is a $\Benb_n$-pair.  
    \item There exists a $\Delta_{n+1}(\PA)$ sentence which is simultaneously non-trivially hereditarily $\mathcal{B}(\Sigma_n)$-conservative over $T$ and $U$. 
    \item There exists a sentence which is simultaneously exactly hereditarily $\mathcal{B}(\Sigma_n)$-conservative over $T$ and $U$. 
    \item There exists a sentence which is simultaneously non-trivially hereditarily $\mathcal{B}(\Sigma_n)$-conservative over $T$ and $U$. 
    \item $\HCons(\mathcal{B}(\Sigma_n), U) \setminus \Th(T) \neq \emptyset$. 
\end{enumerate}
\end{thm}
\begin{proof}
$(1 \Rightarrow 2)$: Suppose that $(T, U)$ is a $\Benb_n$-pair. 
Let $\varphi_0$ and $\varphi_1$ be $\Pi_{n+1}$ sentences satisfying the following equivalences:
\begin{align*}
    \PA & \vdash \varphi_0 \leftrightarrow \neg \Bigl((\PR_T(\gn{\varphi_0 \lor \varphi_1}) \lor \PR_U(\gn{\varphi_0 \lor \varphi_1})) \preccurlyeq \PR_T^{\Sigma_n \land \Pi_n}(\gn{\neg \varphi_0}) \Bigr), \\
    \PA & \vdash \varphi_1 \leftrightarrow \neg \Bigl((\PR_T(\gn{\varphi_0 \lor \varphi_1}) \lor \PR_U(\gn{\varphi_0 \lor \varphi_1})) \preccurlyeq \PR_U^{\Sigma_n \land \Pi_n}(\gn{\neg \varphi_1}) \Bigr).
\end{align*}
Since
\begin{align*}
    \PA \vdash \varphi_0 \leftrightarrow & \Bigl(\neg (\PR_T(\gn{\varphi_0 \lor \varphi_1}) \lor \PR_U(\gn{\varphi_0 \lor \varphi_1})) \\
    & \qquad \lor \bigl(\PR_T^{\Sigma_n \land \Pi_n}(\gn{\neg \varphi_0}) \prec (\PR_T(\gn{\varphi_0 \lor \varphi_1}) \lor \PR_U(\gn{\varphi_0 \lor \varphi_1})) \bigr) \Bigr), 
\end{align*}
we obtain that $\varphi_0$ is $\Delta_{n+1}(\PA)$. 
In the similar way, we also obtain that $\varphi_1$ is $\Delta_{n+1}(\PA)$, and hence so is $\varphi_0 \lor \varphi_1$. 
We prove that $\varphi_0 \lor \varphi_1$ is simultaneously non-trivially hereditarily $\mathcal{B}(\Sigma_n)$-conservative over $T$ and $U$. 

Suppose, towards a contradiction, that at least one of $T$ and $U$ proves $\varphi_0 \lor \varphi_1$. 
Then, we find the least number $p$ among $T$-proofs and $U$-proofs of $\varphi_0 \lor \varphi_1$. 
Without loss of generality, we may assume that $p$ is a $T$-proof. 
Since $T + \varphi_0 \vdash (\PR_T(\gn{\varphi_0 \lor \varphi_1}) \lor \PR_U(\gn{\varphi_0 \lor \varphi_1})) \preccurlyeq \PR_T^{\Sigma_n \land \Pi_n}(\gn{\neg \varphi_0})$, we have $T + \varphi_0 \vdash \neg \varphi_0$, and thus $T \vdash \neg \varphi_0$. 
Since $T \vdash \varphi_0 \lor \varphi_1$, we have $T \vdash \varphi_1$. 
In the similar way, we have $U \vdash \neg \varphi_1$. 
By referring to the definition of $\varphi_1$, we obtain the following equivalences: \begin{align*}
    \PA \vdash \varphi_1 & \leftrightarrow \exists y < \overline{p}\, \Prf_U^{\Sigma_n \land \Pi_n}(\gn{\neg \varphi_1}, y),\\
    & \leftrightarrow \exists y < \overline{p}\,  \exists u, v \leq y\, \bigl(\Sigma_n(u) \land \Pi_n(v) \\
    & \qquad \qquad \qquad \land \True_{\Sigma_n}(u) \land \True_{\Pi_n}(v) \land \Prf_U( u \dot{\land} v \dot{\to} \gn{\neg \varphi_1}, y) \bigr). 
\end{align*}
So, $\neg \varphi_1$ is $\PA$-provably equivalent to some $\mathcal{B}(\Sigma_n)$ sentence. 
This contradicts the consistency of $T + \Th_{\mathcal{B}(\Sigma_n)}(U)$. 
Therefore, we have proved that $T \nvdash \varphi_0 \lor \varphi_1$ and $U \nvdash \varphi_0 \lor \varphi_1$. 

We prove that $\varphi_0 \lor \varphi_1 \in \HCons(\mathcal{B}(\Sigma_n), T)$. 
For this, it is sufficient to prove that $\varphi_0 \lor \varphi_1$ is hereditarily $\Sigma_n \lor \Pi_n$-conservative over $T$. 
Let $V$ be any subtheory of $T$. 
Let $\psi$ be any $\Sigma_n \lor \Pi_n$ sentence such that $V + \varphi_0 \lor \varphi_1 \vdash \psi$. 
Then, $T + \neg \psi \vdash \neg \varphi_0$. 
Since $\neg \psi$ is equivalent to some $\Sigma_n \land \Pi_n$ sentence, we have $V + \neg \psi \vdash \PR_T^{\Sigma_n \land \Pi_n}(\gn{\neg \varphi_0}) \prec (\PR_T(\gn{\varphi_0 \lor \varphi_1}) \lor \PR_U(\gn{\varphi_0 \lor \varphi_1}))$. 
Then, $V + \neg \psi \vdash \varphi_0$. 
Therefore, we conclude $V \vdash \psi$. 
In the same way, $\varphi_0 \lor \varphi_1 \in \HCons(\mathcal{B}(\Sigma_n), U)$ is also proved.

\medskip

$(2 \Rightarrow 3)$: This is because every $\Delta_{n+1}(\PA)$ sentence which is simultaneously non-trivially hereditarily $\mathcal{B}(\Sigma_n)$-conservative is simultaneously exactly hereditarily $\mathcal{B}(\Sigma_n)$-conservative. 

\medskip

$(3 \Rightarrow 4)$ and $(4 \Rightarrow 5)$: Obvious. 

\medskip

$(5 \Rightarrow 1)$: This is a consequence of Lemma \ref{Lem_useful}. 
\end{proof}

Second, we discuss simultaneous $\mathcal{B}(\Sigma_n)$-conservativity. 

\begin{thm}\label{Thm_Boole}
For any theories $T$ and $U$, there exists a $\Delta_{n+1}(\PA)$ sentence which is simultaneously non-trivially $\mathcal{B}(\Sigma_n)$-conservative over $T$ and $U$. 
\end{thm}
\begin{proof}
Let $T$ and $U$ be any theories.
If $(T, U)$ is a $\Benb_n$-pair, then by Theorem \ref{Thm_Boole_H}, we are done. 
If $(T, U)$ is not a $\Benb_n$-pair, then there exists a $\mathcal{B}(\Sigma_n)$ sentence $\varphi$ such that $T \vdash \varphi$ and $U \vdash \neg \varphi$. 
By the Fixed Point Theorem, let $\theta_0$ and $\theta_1$ be $\Pi_{n+1}$ sentences such that 
\begin{align*}
    \PA & \vdash \theta_0 \leftrightarrow \varphi \wedge \neg (\PR_T(\gn{\theta_0 \vee \theta_1}) \preccurlyeq \PR_{T}^{\Sigma_n \wedge \Pi_n}(\gn{\neg \theta_0})), \\
    \PA & \vdash \theta_1 \leftrightarrow \neg \varphi \wedge \neg (\PR_U(\gn{\theta_0 \vee \theta_1}) \preccurlyeq \PR_{U}^{\Sigma_n \wedge \Pi_n}(\gn{\neg \theta_1})).
\end{align*}
As in the proof of Theorem \ref{Thm_Boole_H}, it is shown that both $\theta_0$ and $\theta_1$ are $\Delta_{n+1}(\PA)$. 

We prove $T \nvdash \theta_0 \vee \theta_1$. Suppose, towards a contradiction, that $T \vdash \theta_0 \vee \theta_1$. 
Then, we have $T + \theta_0 \vdash \PR_T(\gn{\theta_0 \vee \theta_1}) \preccurlyeq \PR_{T}^{\Sigma_n \wedge \Pi_n}(\gn{\neg \theta_0})$ and it follows that $T + \theta_0 \vdash \neg \theta_0$. 
Hence $T \vdash \neg \theta_0$, which implies $T \vdash \theta_1$. 
Thus, $T \vdash \neg \varphi$, which contradicts the consistency of $T$. 

We prove $\theta_0 \vee \theta_1 \in \Cons(\mathcal{B}(\Sigma_n), T)$.
It suffices to prove that $\theta_0 \vee \theta_1 $ is 
$\Sigma_n \vee \Pi_n$-conservative over $T$.
Let $\psi$ be any $\Sigma_n \lor \Pi_n$ sentence such that $T + \theta_0 \vee \theta_1 \vdash \psi$. 
Then, $T + \theta_0 \vdash \psi$, and we obtain $T + \neg \psi \vdash \neg \theta_0$.
Since $\neg \psi$ is equivalent to some $\Sigma_n \land \Pi_n$ sentence and $T \nvdash \theta_0 \lor \theta_1$, we have $\PA + \neg \psi \vdash \PR_{T}^{\Sigma_n \wedge \Pi_n}(\gn{\neg \theta_0}) \prec \PR_T(\gn{\theta_0 \lor \theta_1})$.
Then, $T+ \neg \psi \vdash \neg (\PR_T(\gn{\theta_0 \vee \theta_1}) \preccurlyeq \PR_{T}^{\Sigma_n \wedge \Pi_n}(\gn{\neg \theta_0}))$. 
Since $T \vdash \varphi$, we have $T + \neg \psi \vdash \theta_0$.
Since $T + \theta_0 \vdash \psi$, we conclude that $T \vdash \psi$.

In the same way, $U \nvdash \theta_0 \vee \theta_1$ and $\theta_0 \vee \theta_1 \in \Cons(\mathcal{B}(\Sigma_n), U)$ are also proved.
\end{proof}

\section{$\Sigma_n \land \Pi_n$-conservativity}\label{Sec:and}

The situation with respect to $\Sigma_n \land \Pi_n$-conservativity is also simple as in the case of $\mathcal{B}(\Sigma_n)$-conservativity. 

First, we discuss simultaneous hereditary $\Sigma_n \land \Pi_n$-conservativity. 
We show that simultaneous non-trivial hereditary $\Sigma_n \land \Pi_n$-conservativity and simultaneous exact hereditary $\Sigma_n \land \Pi_n$-conservativity are equivalent. 
In particular, the notion of $\Bena_n$-pairs that is defined in Definition \ref{Bena_pair} plays an important role. 
Before proving our characterization theorem of the existence of a sentence which is simultaneously exactly hereditarily $\Sigma_n \land \Pi_n$-conservative, we prove the following theorem. 

\begin{thm}\label{Thm_and_char_H}
For any theories $T$ and $U$ and any $n \geq 1$, the following are equivalent: 
\begin{enumerate}
    \item $T + \Th_{\Sigma_n \land \Pi_n}(U)$ is consistent. 
    \item $\mathcal{B}(\Sigma_n) \cap \HCons(\Sigma_n \land \Pi_n, T) \cap \HCons(\Sigma_n \land \Pi_n, U) \setminus \Th(T) \neq \emptyset$. 
    \item $\HCons(\Sigma_n \land \Pi_n, U) \setminus \Th(T) \neq \emptyset$. 
\end{enumerate}
\end{thm}
\begin{proof}
$(1 \Rightarrow 2)$: Suppose that $T + \Th_{\Sigma_n \land \Pi_n}(U)$ is consistent. 
Then, $(T + \Th_{\Pi_n}(U)) + \Th_{\Sigma_n}(U)$ is consistent. 
Equivalently, $U + \Th_{\Pi_n}(T + \Th_{\Pi_n}(U))$ is consistent. 
By Theorem \ref{Thm_char_H_refine}, we find a $\Pi_n$ sentence $\varphi$ such that
\[
    \varphi \in \HCons(\Sigma_n, T) \cap \HCons(\Sigma_n, U) \setminus \Th(T + \Th_{\Pi_n}(U)). 
\]

\paragraph{Claim.}
The theory $(T + \neg \varphi) + \Th_{\Pi_n}(U)$ is consistent. 

\begin{proof}[Proof of Claim.]
Suppose, towards a contradiction, that $(T + \neg \varphi) + \Th_{\Pi_n}(U)$ is inconsistent. 
Then, there exists a $\Pi_n$ sentence $\chi$ such that $U \vdash \chi$ and $T + \neg \varphi \vdash \neg \chi$. 
Since $\chi \in \Th_{\Pi_n}(U)$ and $T + \chi \vdash \varphi$, we obtain $T + \Th_{\Pi_n}(U) \vdash \varphi$. 
This is a contradiction. 
\end{proof}

By Theorem \ref{Thm_char_H_refine} again, we get a $\Sigma_n$ sentence $\psi$ such that
\[
    \psi \in \HCons(\Pi_n, T) \cap \HCons(\Pi_n, U) \setminus \Th(T + \neg \varphi). 
\]
We prove that the $\mathcal{B}(\Sigma_n)$ sentence $\varphi \lor \psi$ is in $\HCons(\Sigma_n \land \Pi_n, T) \cap \HCons(\Sigma_n \land \Pi_n, U) \setminus \Th(T)$. 
It follows from $T + \neg \varphi \nvdash \psi$ that $T \nvdash \varphi \lor \psi$. 

Let $V$ be any subtheory of $T$ or $U$. 
Let $\sigma \in \Sigma_n$ and $\pi \in \Pi_n$ be such that $V + \varphi \lor \psi \vdash \sigma \land \pi$. 
Then, $V + \varphi \vdash \sigma$ and $V + \psi \in \pi$. 
By the choices of $\varphi$ and $\psi$, we conclude $V \vdash \sigma \land \pi$.

$(2 \Rightarrow 3)$: Trivial. 

$(3 \Rightarrow 1)$: This implication directly follows from Lemma \ref{Lem_useful}.
\end{proof}

\begin{thm}\label{Thm_and_H}
For any theories $T$ and $U$ and any $n \geq 1$, the following are equivalent: 
\begin{enumerate}
    \item $(T, U)$ is a $\Bena_n$-pair. 
    \item There exists a $\mathcal{B}(\Sigma_n)$ sentence which is simultaneously non-trivially hereditarily $\Sigma_n \land \Pi_n$-conservative over $T$ and $U$. 
    \item There exists a sentence which is simultaneously exactly hereditarily $\Sigma_n \land \Pi_n$-conservative over $T$ and $U$. 
    \item There exists a sentence which is simultaneously non-trivially hereditarily $\Sigma_n \land \Pi_n$-conservative over $T$ and $U$. 
\end{enumerate}
\end{thm}
\begin{proof}
$(1 \Rightarrow 2)$:
The proof of this implication is almost the same as that of $(1 \Rightarrow 2)$ of Theorem \ref{Thm_and_char_H}. 
Since both $T + \Th_{\Sigma_n \land \Pi_n}(U)$ and $U + \Th_{\Sigma_n \land \Pi_n}(T)$ are consistent, the theories
\[
    (T + \Th_{\Pi_n}(U)) + \Th_{\Pi_n}(U + \Th_{\Pi_n}(T))\ \text{and}\ (T + \Th_{\Pi_n}(U)) + \Th_{\Pi_n}(T + \Th_{\Pi_n}(U))
\]
are consistent. 
By Bennet's Theorem \ref{Thm_Ben_H}, we find a $\Pi_n$ sentence $\varphi$ such that
\[
    \varphi \in \HCons(\Sigma_n, T) \cap \HCons(\Sigma_n, U) \setminus (\Th(T + \Th_{\Pi_n}(U)) \cup \Th(U + \Th_{\Pi_n}(T))). 
\]
Then, it is shown that the theories $(T + \neg \varphi) +  \Th_{\Pi_n}(U + \neg \varphi)$ and $(U + \neg \varphi) + \Th_{\Pi_n}(T + \neg \varphi)$ are consistent. 
By Bennet's Theorem \ref{Thm_Ben_H} again, we find a $\Sigma_n$ sentence $\psi$ such that
\[
    \psi \in \HCons(\Pi_n, T) \cap \HCons(\Pi_n, U) \setminus (\Th(T + \neg \varphi) \cup \Th(U + \neg \varphi)). 
\]
As in the proof of Theorem \ref{Thm_and_char_H}, it is proved that the $\mathcal{B}(\Sigma_n)$ sentence $\varphi \lor \psi$ is simultaneously non-trivially hereditarily $\Sigma_n \land \Pi_n$-conservative over $T$ and $U$. 

\medskip

$(2 \Rightarrow 3)$: This is because every $\mathcal{B}(\Sigma_n)$ sentence which is simultaneously non-trivially hereditarily $\Sigma_n \land \Pi_n$-conservative is simultaneously exactly hereditarily $\Sigma_n \land \Pi_n$-conservative. 

\medskip

$(3 \Rightarrow 4)$: Obvious. 

\medskip

$(4 \Rightarrow 1)$: This directly follows from Theorem \ref{Thm_and_char_H}.
\end{proof}

Second, we investigate simultaneous non-trivial $\Sigma_n \land \Pi_n$-conservativity.

\begin{thm}\label{Thm_and}
For any theories $T$ and $U$, there exists a $\mathcal{B}(\Sigma_n)$ sentence which is simultaneously non-trivially $\Sigma_n \land \Pi_n$-conservative over $T$ and $U$. 
\end{thm}
\begin{proof}
Let $T$ and $U$ be any theories. 
First, we show that there always exists a sentence $\varphi$ such that 
\begin{align}\label{and_phi}
    \varphi \in \mathcal{B}(\Sigma_n) \cap \Cons(\Sigma_n \land \Pi_n, T) \setminus \Th(U).
\end{align}
If $U + \Th_{\Sigma_n \land \Pi_n}(T)$ is consistent, then by Theorem \ref{Thm_and_char_H}, there exists $\varphi \in \mathcal{B}(\Sigma_n) \cap \HCons(\Sigma_n \land \Pi_n, T) \setminus \Th(U)$. 
If $U + \Th_{\Sigma_n \land \Pi_n}(T)$ is inconsistent, then there exists a $\Sigma_n \land \Pi_n$ sentence $\varphi$ such that $T \vdash \varphi$ and $U \vdash \neg \varphi$. 
Then, $\varphi \in \mathcal{B}(\Sigma_n) \cap \Cons(\Sigma_n \land \Pi_n, T) \setminus \Th(U)$. 
In either case, we find a sentence $\varphi$ satisfying (\ref{and_phi}).

In the same way, we find a sentence $\psi \in \mathcal{B}(\Sigma_n) \cap \Cons(\Sigma_n \land \Pi_n, U) \setminus \Th(T)$.
By the Fixed Point Theorem, let $\theta_0$ and $\theta_1$ be $\mathcal{B}(\Sigma_n)$ sentences satisfying the following equivalences:
\begin{align*}
    \PA & \vdash \theta_0 \leftrightarrow \varphi \land \Bigl(\PR_T^{\Sigma_n}(\gn{\neg \theta_0}) \preccurlyeq \PR_T(\gn{\theta_0 \lor \theta_1}) \\
    & \qquad \qquad  \qquad \qquad \lor \neg \bigl(\PR_T(\gn{\theta_0 \lor \theta_1}) \preccurlyeq \PR_T^{\Pi_n}(\gn{\neg \theta_0})  \bigr) \Bigr), \\
    \PA & \vdash \theta_1 \leftrightarrow \psi \land \Bigl(\PR_U^{\Sigma_n}(\gn{\neg \theta_1}) \preccurlyeq \PR_U(\gn{\theta_0 \lor \theta_1}) \\
    & \qquad \qquad  \qquad \qquad \lor \neg \bigl(\PR_U(\gn{\theta_0 \lor \theta_1}) \preccurlyeq \PR_U^{\Pi_n}(\gn{\neg \theta_1}) \bigr) \Bigr).
\end{align*}
We prove that the $\mathcal{B}(\Sigma_n)$ sentence $\theta_0 \lor \theta_1$ is simultaneously non-trivially $\Sigma_n \land \Pi_n$-conservative over $T$ and $U$. 

First, we prove $T \nvdash \theta_0 \lor \theta_1$. 
Suppose, towards a contradiction, that $T \vdash \theta_0 \lor \theta_1$. 
Then, we have 
\[
    T + \theta_0 \vdash (\PR_T(\gn{\theta_0 \lor \theta_1}) \prec \PR_T^{\Sigma_n}(\gn{\neg \theta_0})) \land (\PR_T(\gn{\theta_0 \lor \theta_1}) \preccurlyeq \PR_T^{\Pi_n}(\gn{\neg \theta_0})). 
\]
By the choice of $\theta_0$, we get $T + \theta_0 \vdash \neg \theta_0$ and hence $T \vdash \neg \theta_0$. 
By combining this with the supposition, $T \vdash \theta_1$. 
By the choice of $\theta_1$, we have $T \vdash \psi$. 
This is a contradiction. 
Hence, $T \nvdash \theta_0 \lor \theta_1$. 
In the same way, $U \nvdash \theta_0 \lor \theta_1$ is proved. 

Finally, we prove that $\theta_0 \lor \theta_1$ is $\Sigma_n \land \Pi_n$-conservative over $T$. 
Let $\sigma \in \Sigma_n$ and $\pi \in \Pi_n$ be such that $T + \theta_0 \lor \theta_1 \vdash \sigma \land \pi$. 
Since $T + \neg \sigma \vdash \neg \theta_0$ and $T + \neg \pi \vdash \neg \theta_0$, we have $\PA + \neg \sigma \vdash \PR_T^{\Pi_n}(\gn{\neg \theta_0}) \prec \Pr_T(\gn{\theta_0 \lor \theta_1})$ and $\PA + \neg \pi \vdash \PR_T^{\Sigma_n}(\gn{\neg \theta_0}) \preccurlyeq \PR_T(\gn{\theta_0 \lor \theta_1})$. 
By the definition of $\theta_0$, we obtain that $\PA + \neg \sigma \lor \neg \pi + \varphi \vdash \theta_0$. 
Since $T + \theta_0 \vdash \sigma \land \pi$, we get $T + \varphi \vdash \sigma \land \pi$. 
Then, we conclude $T \vdash \sigma \land \pi$ because $\varphi \in \Cons(\Sigma_n \land \Pi_n, T)$. 
In the similar way, $\theta_0 \lor \theta_1 \in \Cons(\Sigma_n \land \Pi_n, U)$ is also proved. 
\end{proof}

\section{Summary of our results}\label{Sec:Summary}

In this section, we summarize our results obtained in the previous sections. 
We have introduced several notions on pairs of theories. 

\begin{itemize}
    \item $(T, U)$ is a $\Ben_n$-pair $\iff$ ($\Th_{\Pi_n}(T) \nsubseteq \Th(U)$ or $T + \Th_{\Pi_n}(U)$ is consistent) and ($\Th_{\Pi_n}(U) \nsubseteq \Th(T)$ or $U + \Th_{\Pi_n}(T)$ is consistent). \hfill (Definition \ref{Bennet_pair})

    \item $(T, U)$ is a $\Bend_n$-pair $\iff$ at least one of $T + \Th_{\Pi_n}(U)$ and $U + \Th_{\Pi_n}(T)$ is consistent. \hfill (Definition \ref{Bend_pair})

    \item $(T, U)$ is a $\Benp_n$-pair $\iff$ $T + \Th_{\Pi_n}(U)$ and $U + \Th_{\Pi_n}(T)$ are consistent. \hfill (Definition \ref{Bennet_pair})

    \item $(T, U)$ is a $\Bena_n$-pair $\iff$ $T + \Th_{\Sigma_n \land \Pi_n}(U)$ and $U + \Th_{\Sigma_n \land \Pi_n}(T)$ are consistent. \hfill (Definition \ref{Bena_pair})

    \item $(T, U)$ is a $\Benb_n$-pair $\iff$ $T + \Th_{\mathcal{B}(\Sigma_n)}(U)$ is consistent. \hfill (Definition \ref{Benb_pair})
\end{itemize}

The implications between the conditions are visualized in Figure \ref{Fig2}, where (i) and (ii) are explained below.   

\begin{figure}[ht]
\centering
\begin{tikzpicture}
\node (Bd+1) at (-1.5,0) {$\Bend_{n+1}$};
\node (BB) at (0,0) {$\Benb_n$};
\node (Ba) at (1,0) {$\Bena_n$};
\node (i) at (2,0) {(i)};
\node (Bp) at (3,0) {$\Benp_n$};
\node (B) at (5,-1) {$\Ben_n$};
\node (ii) at (4,0) {(ii)};
\node (Bd) at (5,0) {$\Bend_n$};

\draw [->, double] (Bd+1)--(BB);
\draw [->, double] (BB)--(Ba);
\draw [->, double] (Ba)--(i);
\draw [->, double] (i)--(Bp);
\draw [->, double] (Bp)--(ii);
\draw [->, double] (ii)--(Bd);
\draw [->, double] (Bp)--(B);
\end{tikzpicture}
\caption{Implications between the conditions}\label{Fig2}
\end{figure}

The conditions of type $\Benp$ are linearly ordered in Figure \ref{Fig2}, whereas the conditions of type $\Ben$ are in a different situation than those of type $\Benp$, as the next two propositions show.
This is the reason why we have distinguished between $\Ben$ and $\Benp$.

\begin{prop}\label{Prop_BC}
Let $n \geq 1$.  
For any theories $T$ and $U$, we have that $(T, U)$ is a $\Ben_n$-pair or a $\Bend_n$-pair. 
\end{prop}
\begin{proof}
Suppose that $(T, U)$ is not a $\Bend_n$-pair, that is, both $T + \Th_{\Pi_n}(U)$ and $U + \Th_{\Pi_n}(T)$ are inconsistent. 
Then, we have $\Th_{\Pi_n}(U) \nsubseteq \Th(T)$ and $\Th_{\Pi_n}(T) \nsubseteq \Th(U)$. 
Thus, $(T, U)$ is a $\Ben_n$-pair. 
\end{proof}

\begin{prop}\label{Prop_Ben}
For any theories $T$ and $U$, there is at most one $n \geq 1$ such that $(T, U)$ is not a $\Ben_n$-pair. 
\end{prop}
\begin{proof}
Suppose that there is an $n \geq 1$ such that $(T, U)$ is not a $\Ben_n$-pair. 
Let $n_0$ be the least such an $n$. 
Then, for any $i > 0$, we have that $(T, U)$ is not a $\Bend_{n+i}$-pair. 
It follows from Proposition \ref{Prop_BC} that $(T, U)$ is a $\Ben_{n+i}$-pair. 
Hence, for all $m \geq 1$ with $m \neq n_0$, we obtain that $(T, U)$ is a $\Ben_m$-pair. 
\end{proof}

\begin{table}[ht]
\centering
\scriptsize{
\begin{tabular}{|c||c|c|c|c|c|}
\hline
\diagbox{$\Gamma$}{$\Theta$} & $\Pi_n$ & $\Sigma_n$ & $\Sigma_n \land \Pi_n$ & $\mathcal{B}(\Sigma_n)$ & $\Delta_{n+1}(\PA)$ \\
\hline
\hline
$\mathcal{B}(\Sigma_n)$ & \diagbox[height=2\line, dir=NE]{}{} & \diagbox[height=2\line, dir=NE]{}{} & \diagbox[height=2\line, dir=NE]{}{} & \diagbox[height=2\line, dir=NE]{}{} & \begin{tabular}{c} $\Benb_n$ \\ Thm.~\ref{Thm_Boole_H} \end{tabular} \\

\hline 
$\Sigma_n \land \Pi_n$ & \diagbox[height=2\line, dir=NE]{}{} & \diagbox[height=2\line, dir=NE]{}{} & \diagbox[height=2\line, dir=NE]{}{} & \begin{tabular}{c} $\Bena_n$ \\ Thm.~\ref{Thm_and_H} \end{tabular} & \begin{tabular}{c} $\Bena_n$ \\ Thm.~\ref{Thm_and_H} \end{tabular} \\

\hline 
$\Pi_n$ & \diagbox[height=3\line, dir =NE]{}{} & \begin{tabular}{c} $\Benp_n$ \\ Thm.~\ref{Thm_Ben_H} \\ Bennet~\cite{Ben86,Ben} \end{tabular} & \begin{tabular}{c} $\Benp_n$ \\ Thm.~\ref{Thm_Pi_H} \end{tabular} & \begin{tabular}{c} $\Benp_n$ \\ Thm.~\ref{Thm_Pi_H} \end{tabular} & \begin{tabular}{c} $\Benp_n$ \\ Thm.~\ref{Thm_Pi_H} \end{tabular} \\

\hline 
$\Sigma_n$ & \begin{tabular}{c} $\Benp_n$ \\ Thm.~\ref{Thm_Ben} \\ Bennet~\cite{Ben86,Ben} \end{tabular} & \diagbox[height=3\line, dir =NE]{}{} & \begin{tabular}{c} $\Benp_n$ \\ Thm.~\ref{Thm_Sigma_H} \end{tabular} & \begin{tabular}{c} $\Benp_n$ \\ Thm.~\ref{Thm_Sigma_H} \end{tabular} & \begin{tabular}{c} $\Benp_n$ \\ Thm.~\ref{Thm_Sigma_H} \end{tabular} \\

\hline
$\SP$ & \begin{tabular}{c} $\times$ \\ Prop.~\ref{Prop_SP_2}.(2) \end{tabular} & \begin{tabular}{c} $\times$ \\ Prop.~\ref{Prop_SP_2}.(1) \end{tabular} & \begin{tabular}{c} $\Benp_n$ \\ Thm.~\ref{Thm_SP_H} \end{tabular} & \begin{tabular}{c} $\Benp_n$ \\ Thm.~\ref{Thm_SP_H} \end{tabular} & \begin{tabular}{c} $\Benp_n$ \\ Thm.~\ref{Thm_SP_H} \end{tabular}\\

\hline 
$\Delta_{n}$ & \begin{tabular}{c} $\times$ \\ Prop.~\ref{Prop_D_1} \end{tabular} & \begin{tabular}{c} (i) \\ Thm.~\ref{Thm_D_H} \end{tabular} & \begin{tabular}{c} (i) \\ Thm.~\ref{Thm_D_H} \end{tabular} & \begin{tabular}{c} (i) \\ Thm.~\ref{Thm_D_H} \end{tabular} & \begin{tabular}{c} (i) \\ Thm.~\ref{Thm_D_H} \end{tabular}\\

\hline
\end{tabular}
}
\caption{The existence of a $\Theta$ sentence which is simultaneously exactly hereditarily $\Gamma$-conservative over $T$ and $U$}\label{table1}
\end{table}

Table \ref{table1} summarizes the status of the existence of a $\Theta$ sentence which is simultaneously exactly hereditarily $\Gamma$-conservative over $T$ and $U$. 
For each $\Gamma$, we define $\Gamma^\ast$ as follows: 
\[
    \Gamma^\ast : = \begin{cases}
    \Delta_{n+1}(\PA) & \text{if}\ \Gamma = \mathcal{B}(\Sigma_n), \\
    \mathcal{B}(\Sigma_n) & \text{if}\ \Gamma = \Sigma_n \land \Pi_n, \\
    \Sigma_n & \text{if}\ \Gamma = \Pi_n, \\
    \Pi_n & \text{if}\ \Gamma = \Sigma_n, \\
    \Sigma_n \land \Pi_n & \text{if}\ \Gamma = \SP, \\
    \Sigma_n & \text{if}\ \Gamma = \Delta_n. 
    \end{cases}
\]
As a consequence of Theorems \ref{Thm_Pi_H}, \ref{Thm_Sigma_H}, \ref{Thm_SP_H}, \ref{Thm_D_H}, \ref{Thm_Boole_H}, and \ref{Thm_and_H}, we obtain the following corollary. 

\begin{cor}\label{cor_eHCons}
For any theories $T$ and $U$ and any $n \geq 1$, if $(T, U)$ is a $\Benb_n$-pair, then for any $\Gamma \in \{\Delta_n, \SP, \Sigma_n, \Pi_n, \Sigma_n \land \Pi_n, \mathcal{B}(\Sigma_n)\}$, there exists a $\Gamma^\ast$ sentence which is simultaneously exactly hereditarily $\Gamma$-conservative over $T$ and $U$. 
\end{cor}

A sufficient condition for each item marked as (i) in Table \ref{table1} is $\Bena_n$ and a necessary condition for (i) is $\Benp_n$ (Theorem \ref{Thm_D_H}). 
However, we could not find a necessary and sufficient condition for each of them (Problem \ref{Prob_D_H}).
Furthermore, we do not know if these items are mutually equivalent or not.

\begin{table}[ht]
\centering
\scriptsize{
\begin{tabular}{|c||c|c|c|c|c|}
\hline
\diagbox{$\Gamma$}{$\Theta$} & $\Pi_n$ & $\Sigma_n$ & $\Sigma_n \land \Pi_n$ & $\mathcal{B}(\Sigma_n)$ & $\Delta_{n+1}(\PA)$ \\
\hline
\hline
$\mathcal{B}(\Sigma_n)$ & \diagbox[height=2\line, dir=NE]{}{} & \diagbox[height=2\line, dir=NE]{}{} & \diagbox[height=2\line, dir=NE]{}{} & \diagbox[height=2\line, dir=NE]{}{} & \begin{tabular}{c} $\Benb_n$ \\ Thm.~\ref{Thm_Boole_H} \end{tabular} \\

\hline 
$\Sigma_n \land \Pi_n$ & \diagbox[height=2\line, dir=NE]{}{} & \diagbox[height=2\line, dir=NE]{}{} & \diagbox[height=2\line, dir=NE]{}{} & \begin{tabular}{c} $\Bena_n$ \\ Thm.~\ref{Thm_and_H} \end{tabular} & \begin{tabular}{c} $\Bena_n$ \\ Thm.~\ref{Thm_and_H} \end{tabular} \\

\hline 
$\Pi_n$ & \diagbox[height=3\line, dir =NE]{}{} & \begin{tabular}{c} $\Benp_n$ \\ Thm.~\ref{Thm_Ben_H} \\ Bennet~\cite{Ben86,Ben} \end{tabular} & \begin{tabular}{c} $\Benp_n$ \\ Thm.~\ref{Thm_Pi_H} \end{tabular} & \begin{tabular}{c} $\Benp_n$ \\ Thm.~\ref{Thm_Pi_H} \end{tabular} & \begin{tabular}{c} $\Benp_n$ \\ Thm.~\ref{Thm_Pi_H} \end{tabular} \\

\hline 
$\Sigma_n$ & \begin{tabular}{c} $\Benp_n$ \\ Thm.~\ref{Thm_Ben} \\ Bennet~\cite{Ben86,Ben} \end{tabular} & \diagbox[height=3\line, dir =NE]{}{} & \begin{tabular}{c} $\Benp_n$ \\ Thm.~\ref{Thm_Sigma_H} \end{tabular} & \begin{tabular}{c} $\Benp_n$ \\ Thm.~\ref{Thm_Sigma_H} \end{tabular} & \begin{tabular}{c} $\Benp_n$ \\ Thm.~\ref{Thm_Sigma_H} \end{tabular} \\

\hline
$\SP$ & \begin{tabular}{c} $\Benp_n$ \\ Cor.~\ref{Cor_SP_ntH} \end{tabular} & \begin{tabular}{c} $\Benp_n$ \\ Cor.~\ref{Cor_SP_ntH} \end{tabular} & \begin{tabular}{c} $\Bend_n$ \\ Thm.~\ref{Thm_SP_ntH} \end{tabular} & \begin{tabular}{c} $\Bend_n$ \\ Thm.~\ref{Thm_SP_ntH} \end{tabular} & \begin{tabular}{c} $\Bend_n$ \\ Thm.~\ref{Thm_SP_ntH} \end{tabular}\\

\hline 
$\Delta_{n}$ & \begin{tabular}{c} $\Benp_n$ \\ Cor.~\ref{Cor_D_ntH} \end{tabular} & \begin{tabular}{c} (ii) \\ Thm.~\ref{Thm_D_ntH} \end{tabular} & \begin{tabular}{c} $\Bend_n$ \\ Thm.~\ref{Thm_D_ntH} \end{tabular} & \begin{tabular}{c} $\Bend_n$ \\ Thm.~\ref{Thm_D_ntH} \end{tabular} & \begin{tabular}{c} $\Bend_n$ \\ Thm.~\ref{Thm_D_ntH} \end{tabular} \\

\hline
\end{tabular}
}
\caption{The existence of a $\Theta$ sentence which is simultaneously non-trivially hereditarily $\Gamma$-conservative over $T$ and $U$}\label{table2}
\end{table}

Table \ref{table2} summarizes the status of the existence of a $\Theta$ sentence which is simultaneously non-trivially hereditarily $\Gamma$-conservative over $T$ and $U$. 
This table seems interesting in the sense that a variety of conditions ($\Benb_n$, $\Bena_n$, $\Benp_n$, and $\Bend_n$) appear.
The existence of a $\Sigma_n$ sentence which is simultaneously non-trivially $\Delta_n$-conservative over $T$ and $U$ is marked as (ii) in Table \ref{table2}. 
A sufficient condition for (ii) is $\Benp_n$ and a necessary condition for (ii) is $\Bend_n$ (Theorem \ref{Thm_D_ntH}). 
However, we could not also find a necessary and sufficient condition for (ii) (Problem \ref{Prob_D_ntH}).

\begin{table}[ht]
\centering
\scriptsize{
\begin{tabular}{|c||c|c|c|c|c|}
\hline
\diagbox{$\Gamma$}{$\Theta$} & $\Pi_n$ & $\Sigma_n$ & $\Sigma_n \land \Pi_n$ & $\mathcal{B}(\Sigma_n)$ & $\Delta_{n+1}(\PA)$ \\
\hline
\hline
$\mathcal{B}(\Sigma_n)$ & \diagbox[height=2\line, dir=NE]{}{} & \diagbox[height=2\line, dir=NE]{}{} & \diagbox[height=2\line, dir=NE]{}{} & \diagbox[height=2\line, dir=NE]{}{} & \begin{tabular}{c} $\checkmark$ \\ Thm.~\ref{Thm_Boole} \end{tabular} \\

\hline 
$\Sigma_n \land \Pi_n$ & \diagbox[height=2\line, dir=NE]{}{} & \diagbox[height=2\line, dir=NE]{}{} & \diagbox[height=2\line, dir=NE]{}{} & \begin{tabular}{c} $\checkmark$ \\ Thm.~\ref{Thm_and} \end{tabular} & $\checkmark$ \\

\hline 
$\Pi_n$ & \diagbox[height=3\line, dir=NE]{}{} & \begin{tabular}{c} $\checkmark$  \\ Thm.~\ref{Thm_KOSV} \\ KOSV~\cite{KOSV} \end{tabular} & $\checkmark$ & $\checkmark$ & $\checkmark$ \\

\hline 
$\Sigma_n$ & \begin{tabular}{c} $\Ben_n$ \\ Thm.~\ref{Thm_Ben} \\ Bennet~\cite{Ben86,Ben} \end{tabular} & \diagbox[height=3\line, dir=NE]{}{} & \begin{tabular}{c} $\Ben_n$ \\ Thm.~\ref{Thm_Sigma} \end{tabular} & \begin{tabular}{c} $\Ben_n$ \\ Thm.~\ref{Thm_Sigma} \end{tabular} & \begin{tabular}{c} $\Ben_n$ \\ Thm.~\ref{Thm_Sigma} \end{tabular} \\

\hline
$\SP$ & \begin{tabular}{c} $\times$ \\ Prop.~\ref{Prop_SP_2}.(2) \end{tabular} & \begin{tabular}{c} $\times$ \\ Prop.~\ref{Prop_SP_2}.(1) \end{tabular} & \begin{tabular}{c} $\Ben_n$ \\ Thm.~\ref{Thm_SP} \end{tabular} & \begin{tabular}{c} $\Ben_n$ \\ Thm.~\ref{Thm_SP} \end{tabular} & \begin{tabular}{c} $\Ben_n$ \\ Thm.~\ref{Thm_SP} \end{tabular}\\

\hline 
$\Delta_{n}$ & \begin{tabular}{c} $\times$ \\ Prop.~\ref{Prop_D_1} \end{tabular} & \begin{tabular}{c} $\Ben_n$ \\ Thm.~\ref{Thm_D} \end{tabular} & \begin{tabular}{c} $\Ben_n$ \\ Thm.~\ref{Thm_D} \end{tabular} & \begin{tabular}{c} $\Ben_n$ \\ Thm.~\ref{Thm_D} \end{tabular} & \begin{tabular}{c} $\Ben_n$ \\ Thm.~\ref{Thm_D} \end{tabular}\\

\hline
\end{tabular}
}
\caption{The existence of a $\Theta$ sentence which is simultaneously exactly $\Gamma$-conservative over $T$ and $U$}\label{table3}
\end{table}

Table \ref{table3} summarizes the status of the existence of a $\Theta$ sentence which is simultaneously exactly $\Gamma$-conservative over $T$ and $U$. 
We were able to draw a complete picture of this table. 
It is interesting to note that in the table, the situation may be divided into two parts: above $\Pi_n$ and below $\Sigma_n$.

It is proved in Proposition \ref{Prop_BC} that for any theories $T$ and $U$, we have that $(T, U)$ is $\Ben_n$-pair or $\Bend_n$-pair. 
So, we obtain the following corollary to Theorems \ref{Thm_SP_ntH} and \ref{Thm_SP}. 

\begin{cor}
For any theories $T$ and $U$ and any $n \geq 1$, there exists a $\Sigma_n \land \Pi_n$ sentence $\varphi$ satisfying at least one of the following conditions: 
\begin{enumerate}
    \item $\varphi$ is simultaneously non-trivially hereditarily $\SP$-conservative over $T$ and $U$, 
    \item $\varphi$ is simultaneously exactly $\SP$-conservative over $T$ and $U$. 
\end{enumerate}
\end{cor}

\begin{table}[ht]
\centering
\scriptsize{
\begin{tabular}{|c||c|c|c|c|c|}
\hline
\diagbox{$\Gamma$}{$\Theta$} & $\Pi_n$ & $\Sigma_n$ & $\Sigma_n \land \Pi_n$ & $\mathcal{B}(\Sigma_n)$ & $\Delta_{n+1}(\PA)$ \\
\hline
\hline
$\mathcal{B}(\Sigma_n)$ & \diagbox[height=2\line, dir=NE]{}{} & \diagbox[height=2\line, dir=NE]{}{} & \diagbox[height=2\line, dir=NE]{}{} & \diagbox[height=2\line, dir=NE]{}{} & \begin{tabular}{c} $\checkmark$ \\ Thm.~\ref{Thm_Boole} \end{tabular} \\

\hline 
$\Sigma_n \land \Pi_n$ & \diagbox[height=2\line, dir=NE]{}{} & \diagbox[height=2\line, dir=NE]{}{} & \diagbox[height=2\line, dir=NE]{}{} & \begin{tabular}{c} $\checkmark$ \\ Thm.~\ref{Thm_and} \end{tabular} & $\checkmark$ \\

\hline 
$\Pi_n$ & \diagbox[height=3\line, dir=NE]{}{} & \begin{tabular}{c} $\checkmark$  \\ Thm.~\ref{Thm_KOSV} \\ KOSV~\cite{KOSV} \end{tabular} & $\checkmark$ & $\checkmark$ & $\checkmark$ \\

\hline 
$\Sigma_n$ & \begin{tabular}{c} $\Ben_n$ \\ Thm.~\ref{Thm_Ben} \\ Bennet~\cite{Ben86,Ben} \end{tabular} & \diagbox[height=3\line, dir=NE]{}{} & \begin{tabular}{c} $\Ben_n$ \\ Thm.~\ref{Thm_Sigma} \end{tabular} & $\checkmark$ & $\checkmark$ \\

\hline
$\SP$ & \begin{tabular}{c} $\Ben_n$ \\ Cor.~\ref{Cor_SP_nt} \end{tabular} & $\checkmark$ & $\checkmark$ & $\checkmark$ & $\checkmark$ \\

\hline 
$\Delta_{n}$ & \begin{tabular}{c} $\Ben_n$ \\ Cor.~\ref{Cor_D_nt} \end{tabular} & $\checkmark$ & $\checkmark$ & $\checkmark$ & $\checkmark$ \\

\hline
\end{tabular}
}
\caption{The existence of a $\Theta$ sentence which is simultaneously non-trivially $\Gamma$-conservative over $T$ and $U$}\label{table4}
\end{table}

Finally, Table \ref{table4} summarizes the status of the existence of a $\Theta$ sentence which is simultaneously non-trivially $\Gamma$-conservative over $T$ and $U$.

\section{Answering Guaspari's question}\label{Sec:answer}

We close the present paper with answering Guaspari's question (Problem \ref{Gua_Prob}). 
The following corollary immediately follows from Corollary \ref{cor_eHCons} because $(T, T)$ is a $\Benb_n$-pair for every theory $T$. 
Recall that for each $\Gamma$, the formula class $\Gamma^\ast$ is defined in the last section. 

\begin{cor}\label{single_cor}
For any theory $T$ and any $\Gamma$, there exists a $\Gamma^\ast$ sentence which is exactly hereditarily $\Gamma$-conservative over $T$. 
\end{cor}

   




Then, we would like to present an answer to Guaspari's question in a general form by refining Corollary \ref{single_cor} based on the idea of his proof of Guaspari's Theorem \ref{Thm_Gua}. 
However, there is one obstacle that was already discussed in the proof of Theorem \ref{Thm_Sigma}. 

\begin{prop}
Let $n \geq 1$, $T$ be any theory, and $\varphi$ be any sentence. 
If $\varphi$ is a $\Sigma_n \land \Pi_n$ sentence and $\Sigma_n$-conservative (resp.~$\Pi_n$-conservative) over $T$, then $\varphi$ is $T$-provably equivalent to some $\Pi_n$ (resp.~$\Sigma_n$) sentence. 
\end{prop}
\begin{proof}
    Let $\sigma \in \Sigma_n$ and $\pi \in \Pi_n$ be such that $\varphi \equiv \sigma \land \pi$ and $\varphi \in \Cons(\Sigma_n, T)$. 
    Since $T + \varphi \vdash \sigma$, we have $T \vdash \sigma$. 
    Hence, $T \vdash \varphi \leftrightarrow \pi$. 
    The case $\varphi \in \Cons(\Pi_n, T)$ is also proved similarly. 
\end{proof}

So, there is no sentence that is essentially $\Sigma_n \land \Pi_n$ and $\Sigma_n$-conservative or $\Pi_n$-conservative over $T$. 
By excluding these cases, we obtain the following theorem, which is our solution to Guaspari's question.

\begin{thm}\label{MT}
For any theory $T$, any $\Gamma$, and any $\Theta \supseteq \Gamma^\ast$, if $(\Gamma, \Theta)$ is neither $(\Sigma_n, \Sigma_n \land \Pi_n)$ nor $(\Pi_n, \Sigma_n \land \Pi_n)$ for all $n \geq 1$, then there exists a sentence which is essentially $\Theta$ and exactly hereditarily $\Gamma$-conservative over $T$. 
\end{thm}
\begin{proof}
We fix any theory $T$, any $\Gamma$, and any $\Theta \supseteq \Gamma^\ast$ such that $(\Gamma, \Theta)$ is neither $(\Sigma_n, \Sigma_n \land \Pi_n)$ nor $(\Pi_n, \Sigma_n \land \Pi_n)$ for all $n \geq 1$.
By Corollary \ref{single_cor}, there exists a $\Gamma^\ast$ sentence $\varphi$ that is exactly hereditarily $\Gamma$-conservative over $T$. 
We define $\Lambda$ depending on $\Theta$ as follows: 
\[
    \Lambda : = \begin{cases}
    \mathcal{B}(\Sigma_n) & \text{if}\ \Theta = \Delta_{n+1}(\PA), \\
    \Sigma_n \land \Pi_n & \text{if}\ \Theta = \mathcal{B}(\Sigma_n), \\
    \SP & \text{if}\ \Theta = \Sigma_n \land \Pi_n, \\
    \Pi_n & \text{if}\ \Theta = \Sigma_n, \\
    \Sigma_n & \text{if}\ \Theta = \Pi_n. 
    \end{cases}
\]
Then, we have that $\Lambda^\ast = \Theta$. 
So, by Corollary \ref{single_cor} again, we find a $\Theta$ sentence $\psi$ which is exactly hereditarily $\Lambda$-conservative over $T + \varphi$. 
Let $\rho$ be the sentence $\varphi \land \psi$. 
We show that $\rho$ is essentially $\Theta$ and exactly hereditarily $\Gamma$-conservative over $T$. 

First, we show that $\rho$ is essentially $\Theta$ (cf.~Definition \ref{essentiality}). 
We have $\rho \in \Theta$ because $\varphi \in \Gamma^\ast$, $\psi \in \Theta$, and $\Gamma^\ast \subseteq \Theta$. 
Let $\Theta'$ be a class with $\Theta \nsubseteq \Theta'$ and suppose that there exists a $\Theta'$ sentence $\theta$ such that $T \vdash \rho \leftrightarrow \theta$. 
We distinguish the following two cases. 

\paragraph{Case 1.} $\Theta \neq \Sigma_n \land \Pi_n$: 
In this case, $\Lambda^\ast  = \Theta \nsubseteq \Theta'$ implies $\Theta' \subseteq \Lambda$. 
Since $T + \varphi + \psi \vdash \theta$ and $\psi \in \Cons(\Lambda, T + \varphi) \subseteq \Cons(\Theta', T + \varphi)$, we have $T + \varphi \vdash \theta$. 
Hence $T + \varphi \vdash \psi$. 
This contradicts the exactness of $\psi$. 

\paragraph{Case 2.} $\Theta = \Sigma_n \land \Pi_n$: 
In this case, $\Lambda = \Sigma_n {\downarrow} \Pi_n$. 
Also, $\Sigma_n \land \Pi_n \nsubseteq \Theta'$ implies $\Theta' \subseteq \Sigma_n$ or $\Theta' \subseteq \Pi_n$. 
If $\Theta' \subseteq \Delta_n(\PA)$, then we have $\psi \in \Cons(\Sigma_n {\downarrow} \Pi_n, T + \varphi) \subseteq \Cons(\Theta', T + \varphi)$ and the proof proceeds as in the proof of Case 1. 
So, we may assume that $\Theta'$ is either $\Sigma_n$ or $\Pi_n$. 

Suppose $\Theta' = \Sigma_n$. 
Then, $T \vdash \rho \leftrightarrow \theta$ implies $T + \varphi \vdash \psi \leftrightarrow \theta$, and thus $\theta$ is a $\Sigma_n$ sentence which is $\Sigma_n {\downarrow} \Pi_n$-conservative over $T + \varphi$. 
By Proposition \ref{Prop_SP_2}.(1), $\theta$ is $\Pi_n$-conservative over $T + \varphi$. 
Hence, $\psi$ is also $\Pi_n$-conservative over $T + \varphi$. 
This contradicts the exactness of $\psi$. 
The case $\Theta' = \Pi_n$ is proved similarly by using Proposition \ref{Prop_SP_2}.(2). 

\medskip

Second, we prove that $\rho$ is exactly hereditarily $\Gamma$-conservative over $T$. 
The exactness of $\rho$ directly follows from the exactness of $\varphi$. 
So, it suffices to prove $\rho \in \HCons(\Gamma, T)$. 
Since $\varphi \in \HCons(\Gamma, T)$, it suffices to show that $\psi \in \HCons(\Gamma, T + \varphi)$. 
We distinguish the following two cases: 

\paragraph{Case 1.} $\Gamma \neq \Sigma_n$ and $\Gamma \neq \Pi_n$: 
In this case, $\Gamma^\ast \subseteq \Theta = \Lambda^\ast$ implies $\HCons(\Lambda, T + \varphi) \subseteq \HCons(\Gamma, T + \varphi)$, and so $\psi \in \HCons(\Gamma, T + \varphi)$. 

\paragraph{Case 2.} $\Gamma = \Sigma_n$ or $\Gamma = \Pi_n$: 
In this case, $\Gamma^\ast \subseteq \Theta = \Lambda^\ast$ implies either $\Lambda = \SP$ or $\HCons(\Lambda, T + \varphi) \subseteq \HCons(\Gamma, T + \varphi)$. 
In the former case, we have $\Theta = \Sigma_n \land \Pi_n$, but this is not the case. 
So, the latter case always holds. 
We have $\psi \in \HCons(\Gamma, T + \varphi)$. 
\end{proof}

\section*{Acknowledgments}

The second author was supported by JSPS KAKENHI Grant Number JP23K03200.

\bibliographystyle{plain}
\bibliography{ref}
\end{document}